\documentclass[10pt,reqno]{amsart}
\usepackage{xcomment}
\usepackage[colorlinks=true, linkcolor=blue, urlcolor=blue, citecolor=blue,%
pagebackref=true]{hyperref}
\usepackage{setspace}

\onehalfspace

\usepackage{tikz}
\usetikzlibrary{matrix}

\usepackage{amssymb}

\usepackage[alphabetic,abbrev,msc-links,nobysame]{amsrefs}

\let\pol\l
\renewcommand{\a}{\alpha}
\renewcommand{\b}{\beta}
\renewcommand{\d}{\delta}

\newcommand{\ve}{\varepsilon}

\newcommand{\f}{\varphi}
\newcommand{\g}{\gamma}
\newcommand{\Ga}{\Gamma}

\renewcommand{\l}{\lambda}

\renewcommand{\o}{\omega}

\newcommand{\s}{\sigma}

\newcommand{\N}{{\mathbb N}}
\newcommand{\Z}{{\mathbb Z}}

\newcommand{\C}{{\bf C}}

\newcommand{\wt}{\widetilde}

\newcommand{\fg}{\mathfrak{g}}

\newcommand\sm{\backslash}

\newcommand{\diam}{{\rm diam}}
\newcommand{\dist}{{\rm dist}}
\newcommand{\Id}{{\rm Id}}
\newcommand{\Stab}{{\rm Stab}}

\DeclareMathOperator{\Pop}{P}
\DeclareMathOperator{\RP}{RP}
\DeclareMathOperator{\HK}{HK}
\DeclareMathOperator{\Aut}{Aut}

\newcommand{\invlim}{\underset{\longleftarrow}{\lim}}

\def\be{\begin{equation}}
\def\ee{\end{equation}}

\theoremstyle{plain}
\newtheorem{lem}{Lemma}[section]
\newtheorem{thm}[lem]{Theorem}
\newtheorem{prp}[lem]{Proposition}

\newtheorem*{fact}{Fact}
\theoremstyle{definition}
\newtheorem{dfn}[lem]{Definition}
\newtheorem{rmk}[lem]{Remark}

\date{\today}
\begin{document}
\title[Inverse limit representations]{The structure theory of Nilspaces III: Inverse limit representations
and topological dynamics}
\author{Yonatan Gutman, Freddie Manners and P\'{e}ter P. Varj\'{u}}

\begin{abstract}
This paper forms the third part of a series by the authors \cites{GMV1,GMV2} concerning the structure theory of  \emph{nilspaces}. A nilspace is a compact space $X$ together with closed collections of \emph{cubes} $C^n(X)\subseteq X^{2^n}$, $n=1,2,\ldots$, satisfying some natural axioms. Our goal is to extend the structure theory of nilspaces obtained by Antol\'\i n Camarena and Szegedy, and to provide new proofs.

Our main result is that, under the technical assumption that $C^n(X)$ is a connected space for all $n$, then $X$ is isomorphic (in a strong sense) to an inverse limit of nilmanifolds.  This is a direct and slight generalization of the main result of Antol\'\i n Camarena and Szegedy.

We also apply our methods to obtain structure theorems in the setting of topological dynamics.  Specifically, if $H$ is a group (subject to very mild topological assumptions) and $(H,X)$ is a minimal dynamical system, then we give a simple characterization of the maximal pronilfactor of $X$.  This generalizes the case $H = \mathbb{Z}$, which is a theorem of Host, Kra and Maass, although even in that case we give a significantly different proof.

\end{abstract}

\address{Yonatan Gutman, Institute of Mathematics, Polish Academy of Sciences,
ul. \'{S}niadeckich~8, 00-656 Warszawa, Poland.}
\email{y.gutman@impan.pl}

\address{Freddie Manners, UCSD Department of Mathematics, 9500 Gilman Drive \#0112, La Jolla, CA 92093, USA}
\email{fmanners@ucsd.edu}

\address{P\'{e}ter P. Varj\'{u}, Centre for Mathematical Sciences,
Wilberforce Road, Cambridge CB3 0WA,
UK}
\email{pv270@dpmms.cam.ac.uk}
\keywords{Nilspace, nilmanifold, nilsystem, system of finite order, Lie groups, regionally proximal relation.}
\subjclass[2010]{Primary 37B05; Secondary 11B30, 54H20.}

\thanks{YG was partially supported by the ERC Grant \emph{Approximate Algebraic Structures and Applications} and the NCN (National Science Center, Poland) grants 2016/22/E/ST1/00448 and 2013/08/A/ST1/00275. PPV was supported by the Royal Society.}

\maketitle

\tableofcontents

\section{Introduction}\label{sc:intro}

This paper is the third and last in a series of papers devoted to the structure theory of \emph{cubespaces} and \emph{nilspaces},
the previous two parts of the project being \cite{GMV1} and \cite{GMV2}.

A cubespace is a structure consisting of a compact metric space $X$, together with a closed collection
of ``cubes'' $C^k(X)\subseteq X^{2^k}$ for each integer $k\ge 0$, satisfying certain axioms that we will recall later.
The structure $(X,\{C^k(X)\}_{k\in\N})$ is further called a nilspace 
if it also satisfies certain extra rigidity conditions.

The notion of nilspaces
has its origins in the work of Host and Kra \cite{HK08}, where these objects appeared under the name of  ``parallelepiped structures''.  The study of these objects was furthered by Antol\'\i n Camarena and Szegedy \cite{CS12}, who in the same work formulated a strong structure theorem for nilspaces, subject to certain further hypotheses.

The papers of Candela \cites{Can1,Can2} expand on \cite{CS12}, providing more detailed proofs. He also includes
several additional results implicit in \cite{CS12}, particularly about continuous systems of measures.

The purpose of our project is to provide a new exposition of this theory, and also
to derive new applications to topological dynamics.
Although we rely heavily on the ideas contained in earlier work \cites{HK05,HM07,CS12,HK08,HKM10,GT10}, our proofs differ from the existing literature in many respects
and we also obtain new results.

The study of nilspaces might be motivated in three different ways.  First, it can be a useful tool in the area of higher order Fourier analysis,
and in particular, forms the basis of Szegedy's approach \cite{S12} to proving an
inverse theorem for the Gowers norms (another approach being due to Green, Tao and Ziegler \cite{GTZ12}).

Second, nilspaces can be used in topological dynamics.
For example, we use them in this paper to generalize a result of Host, Kra and Maass
\cite{HKM10} characterizing the largest pronilfactor of a minimal group action.

Third, nilspaces can be used in the context of ergodic theory, in particular to give a new and more general proof of the structure theorem for \emph{characteristic factors}, introduced by Host and Kra \cite{HK05} and by Ziegler \cite{Z07} (using a different
framework).  For details, see \cite{Gut15,GL19}.

Given the close
connection of the subject to the study of \textit{nilsequences} arising in additive number theory and ergodic theory \cites{BHK05,GT12}, we might expect further applications to arise.

In this paper, we explain the notion of nilspaces, and outline our project, from the point of view of topological dynamics.
The reader whose main interest lies in combinatorics or higher order Fourier analysis may wish
to consult the paper \cite{GMV1}, where an outline is given from that perspective.
She may then continue directly to Section \ref{sec:cubeinvlim} of this paper, where the main result of this paper
is presented in a fashion that does not require familiarity with the dynamical material introduced
in Section \ref{sc:intro}.

\subsection{Regional proximality}\label{sc:RP}

Let $(T,X)$ be a {\bf topological dynamical system}, which for our purposes means that
$X$ is a compact  metric space and $T$ is a homeomorphism on $X$.

We begin by recalling some definitions.
We say that a pair of points $(x,y)\in X^2$ is {\bf proximal}, if there is a sequence
of integers $\{n_i\}$ such that $\lim \dist(T^{n_i}x,T^{n_i}y)=0$.
We say that $(x,y)\in X^2$ is {\bf regionally proximal}, if there are sequences of points
$\{x_i\},\{y_i\}\subseteq X$ and a sequence of integers $\{n_i\}$ such that
$\lim x_i=x$, $\lim y_i=y$ and $\lim \dist(T^{n_i}x_i,T^{n_i}y_i)=0$.
We write $(x,y)\in \Pop_T(X)$ if $(x,y)$ is proximal and $(x,y)\in \RP_T(X)$ if $(x,y)$ is
regionally proximal.

We say that the system $(T,X)$ is {\bf distal} if the proximality relation is trivial; that is, if
$(x,y)\in \Pop_T(X)$ if and only if $x=y$.
We say that the system $(T,X)$ is {\bf minimal} if the orbit of each point is dense.

It turns out that $\RP_T(X)$ is an equivalence relation, and the quotient $X/\RP_T(X)$
has the following property.
\begin{thm}[Ellis and Gottschalk \cite{EG60}*{Theorem 2}]
Let $(T,X)$ be a minimal
topological dynamical system.
Then $\RP_T(X)$ is a closed equivalence relation and $(T,X/\RP_T(X))$ is the maximal
equicontinuous factor of $(T,X)$.

In other words, there is a compact abelian group $K$ and an element $t\in K$ that generates
a dense subgroup of $K$, such that the system $(T,X/\RP_T(X))$ is isomorphic to $(t,K)$.
Moreover, $(T,X/\RP_T(X))$ is the maximal factor with this property.
\end{thm}

Motivated by this result in part, Host, Kra and Maass \cite{HKM10} introduced the notion
of the higher order regional proximal relation,
which can be used analogously to identify the maximal \emph{pronilfactor} of a topological dynamical system.

We recall some definitions.
We call a system $(T,X)$ a {\bf nilsystem} of degree $s$ if there is an $s$-step nilpotent
Lie group $G$, a  discrete cocompact subgroup $\Ga \le G$, and an
element $t\in G$, such that $(T,X)$ is isomorphic as a topological dynamical system to $(t,G/\Ga)$.

Here, by a {\bf Lie group} we mean a Hausdorff, second countable
topological group $G$ equipped with a differentiable structure, such that the map
$G^2\to G:(g,h)\mapsto gh^{-1}$
is differentiable.  We do not assume that Lie groups are connected, and so in particular any countable discrete group is Lie.

Some authors require the Lie
group $G$ in the definition of a nilsystem to be connected (which we do not).
We note that a connected nilmanifold $G/\Ga$ can always be realized in such a way that
$G$ is connected:
indeed, if $G^\circ$ denotes the connected component of the identity in $G$ then we may identify $G/\Ga$ with $G^\circ/(\Ga\cap G^\circ)$.
On the other hand, it may be that a connected nilsystem $(t,G/\Ga)$ cannot be
represented as a nilsystem in which $G$ is connected, if $t \notin G^\circ$.

We write $(T,X_\infty)=\invlim (T,X_i)$ if
the system $(T,X_\infty)$ is the {\bf inverse limit} of the systems $\{(T,X_i)\}$;
i.e., if there is a
family of morphisms (in the category of topological dynamical systems, i.e.~continuous $T$-equivariant maps)
\[\{\f_{i,j}:(T,X_j)\to (T,X_i)\}_{i<j\le\infty}\]
such that $\f_{i,j}\circ\f_{j,l}=\f_{i,l}$ for all triplets of indices $i<j<l\le\infty$,
and such that for any two points $x\neq y\in X_\infty$ there is $i<\infty$ such that $\f_{i,\infty}(x)\neq\f_{i,\infty}(y)$.
A collection of maps $\{\f_{i,j}\}$ with this property is called an {\bf inverse system}.

We say that a system $(T,X)$ is {\bf pronil} of degree $s$, if it is the inverse limit of nilsystems
of degree $s$.  We call a factor $X \to Y$ of $X$ a {\bf pronilfactor} of degree $s$ if $Y$ is pronil of degree $s$,
and say $Y$ is the {\bf maximal pronilfactor} if every other pronilfactor of $X$ factors through $Y$.

The {\bf discrete cube of dimension} $d$ is the set $\{0,1\}^d$.
We write $\vec 0=(0,\ldots,0)\in \{0,1\}^d$ and use the notation $\vec 1$ in a similar manner.
For a vertex $\o=(\o_1,\ldots,\o_d)\in\{0,1\}^d$ and a vector $n=(n_1,\ldots, n_d)\in\Z^d$, we write
$\langle\o, n\rangle=\sum_{i=1}^d\o_i n_i$.
If $G$ is a group and $g=(g_1,\ldots,g_d)\in G^d$, then we write $g^\o=g_1^{\o_1}\cdots g_d^{\o_d}$.

Let $(T,X)$ be a system.
Following Host, Kra and Maass, we say that the pair of points $x,y\in X$
is {\bf regionally proximal of order $s$} if
there are sequences of points $\{x_i\},\{y_i\}\subseteq X$ and a sequence of integer vectors
$\{n_i\}\subseteq\Z^{s}$ such that $\lim x_i = x$, $\lim y_i =y$ and
\[
\lim \dist(T^{\langle\o ,n_i\rangle}x_i,T^{\langle\o ,n_i\rangle} y_i)=0
\]
for all $\o\in\{0,1\}^{s}\sm \{\vec 0\}$.
We denote this relation by $\RP^s_T(X)$.
For a nilsystem $(T,X)$ of degree $s$,  $\RP^s_T(X)$ is trivial.
This fact is non-trivial, we return to it in Section \ref{sc:canonical}
after introducing some additional concepts.

The following theorem of Host, Kra and Maass characterizes the maximal pronilfactor of $(T,X)$ using the higher order regional
proximal relation. \footnote{The case $s=2$ had been proven previously in \cite{HM07}*{Theorem 2}.}

\begin{thm}[\cite{HKM10}*{Theorem 1.3}]\label{th:HKM}
Let $(T,X)$ be a minimal distal system.
Then $\RP^s_T(X)$ is a closed equivalence relation and
$(T,X/\RP^s_T(X))$ is the maximal pronilfactor of $(T,X)$ of degree $s$.
\end{thm}

Host, Kra and Maass also proved that a minimal system $(T,X)$ is distal provided the relation
$\RP_T^s(X)$ is trivial.
Shao and Ye \cite{SY12}*{Theorem 3.5} proved that $\RP^s_T(X)$ is a closed equivalence relation for minimal
(but not necessarily distal)
$(T,X)$  and they combined this with \cite{HKM10} to deduce that Theorem \ref{th:HKM}
holds without the assumption on distality.

The definition of $\RP_T^s$, and those of nilsystems and pronilsystems, generalize straightforwardly to the context of a dynamical system $(H, X)$ where $H$ is any abelian group (that is, $H$ acts continuously on the space $X$), the previous discussion corresponding to the case $H = \Z$.

In \cite{GGY}, Gutman, Glasner and Ye further generalize the definition of the regionally proximal relation to the case of an action by an arbitrary, possibly non-abelian group.
For abelian group actions, in particular for  $\mathbb{Z}$-actions, the new definition coincides with the old one.

We will review this new definition and discuss it in detail in the next subsection, as well as stating a generalization of Theorem \ref{th:HKM} to these more general group actions, which is one of the
main goals of this paper.

\subsection{Regional proximality for non-abelian actions}\label{sc:non-Ab}

We introduce some more definitions.
Let $X$ be a compact metric space and let $H$ be a metrizable topological group acting
continuously on $X$.
We denote the action by $g.x$ for $g\in H$ and $x\in X$ and call the pair
$(H,X)$ a {\bf topological dynamical system} or simply a system.

The topology of $H$ does not play any significant role.
We assume that it is induced by the {\bf maximum displacement metric}
\[
\dist(h_1,h_2)=\max\{\dist(h_1x,h_2x):x\in X\}.
\]

We denote the set of maps $\{0,1\}^\ell\to X$ by $X^{\{0,1\}^\ell}$
and call its elements $\ell$-{\bf configurations}.
Given a configuration $c\in X^{\{0,1\}^\ell}$, we call the points $\{c(\o)\}_{\o\in\{0,1\}^\ell}$
the {\bf vertices} of $c$.
We call a configuration {\bf constant} if all its vertices are equal.

A set of the form $\{\o\in \{0,1\}^\ell:\o_i=\a\}$ for some $1\le i\le d$ and $\a\in\{0,1\}$ is
called a {\bf hyperface} of the discrete cube.
For a hyperface $F$ and an $h\in H$ we denote by $[h]_F$ the element of $H^{\{0,1\}^\ell}$
defined as $[h]_F(\o)=h$ if $\o\in F$ and $[h]_F(\o)=e$ otherwise.
Here and everywhere below $e$ denotes the identity element of the group.
We call the subgroup of $G^{\{0,1\}^\ell}$ generated by
\[
\{[h]_F : h\in H \text{ and $F$ is a hyperface of $\{0,1\}^\ell$}\}
\]
the Host--Kra cube group and denote it by $\HK^\ell(H)$.
These groups originate in \cite{HK05}*{Section 5} and coincide with the \textit{parallelepiped groups} of
\cite{HKM10}*{Definition 3.1} introduced for abelian actions. The terminology is due to
\cite{GT10}*{Definition E.3} where it is employed in the context of filtered
Lie groups.

The Host--Kra cube group acts naturally on the space of $\ell$-configurations on $X$, via
$\g.c(\o)=\g(\o).c(\o)$ for $\g\in \HK^\ell(H)$ and $c\in X^{\{0,1\}^\ell}$.
Following Host, Kra and Maass \cite{HKM10} we call the orbit closure of  constant configurations
the set of {\bf dynamical cubes}\footnote{
In fact, Host, Kra and Maass call these configurations \emph{parallelepipeds},
but we use the term \emph{cubes} in order to conform with \cite{CS12}
and for the sake of brevity.}, denoted
\begin{equation}\label{eq:def of Cn H}
C^\ell_H(X)=\overline{\{\g.x^{\{0,1\}^\ell}:\g\in \HK^\ell(H), x\in X\}} \ .
\end{equation}

If $x,y\in X$ are two points, we write $\llcorner^k(x;y)$ for the configuration
\[
  \o \mapsto \begin{cases} x &\colon \o \ne \vec1 \\ y &\colon \o = \vec1 \end{cases} \ .
\]

We return to the setting of $\Z$-systems as in the previous section, i.e.~we take $H=\{T^n\}_{n\in\Z}$ for a homeomorphism $T$ of $X$.
Host, Kra and Maass \cite{HKM10}*{Corollary 4.3} gave the following
characterization of the regional proximal relation of order $s$
if the system $(H,X)=(\{T^n\},X)$ is minimal and distal:
we have $(x,y)\in \RP^s_T(X)$ if and only if $\llcorner^{s+1}(x;y)\in C^{s+1}_H(X)$.
Shao and Ye \cite{SY12} proved that this holds  for
general abelian actions, even without the assumption that the system is distal.

Motivated by this, in \cite{GGY}, the following definition for the regional
proximal relation for general group actions is introduced.
\begin{dfn}
Let $(H,X)$ be a topological dynamical system.
We say that a pair of points, $x,y\in X$ is {\bf regionally proximal of order} $s$ and write
$(x,y)\in \RP_H^s(X)$,  if and only if  $\llcorner^{s+1}(x;y)\in C^{s+1}_H(X)$.\footnote{The regionally proximal relation of order $s$ is denoted by $\textbf{NRP}^{[s]}(X)$ in \cite{GGY}.}
\end{dfn}

It is shown in \cite{GGY} that, perhaps surprisingly, the newly introduced relation is an
equivalence relation for any minimal action (without any restriction on the group).
Moreover the proof of this more general fact is simpler than the one given in \cite{SY12}.

The nature of $\RP_H^s$ can vary significantly, and we now give two extreme examples.

For the first, note that if $H_{s+1}$ denotes the $(s+1)$-th element of the lower central series of $H$,
then $(x,hx)\in \RP_H^s(X)$ for any $h\in H_{s+1}$.
(For a proof of this fact, see Section \ref{sc:HK}).
Hence if $H$ is perfect, i.e.~$H=[H,H]$, and the action is minimal, then $\RP_H^s(X)=X^2$; in other words, every pair of points is regionally proximal to all orders.

For the second example, let $(H,G/\Ga)$ be a nilsystem in the following generalized sense: let $G$ be an $s$-step nilpotent
Lie group, let $\Ga$ be a discrete cocompact subgroup, and let $H$ act on $G/\Ga$ via a continuous group homomorphism $\phi:H\to G$, i.e.~$(h, x \Ga) \mapsto \phi(h) x \Ga$.  Then it turns out that $\RP_H^t(G/\Ga)$ is the trivial relation for all $t \ge s$ (we will see a proof of this in Section \ref{sc:canonical}).

We are now ready to state a significant generalization of Theorem \ref{th:HKM} to general $H$-actions.

\begin{thm}\label{th:dynamics}
Let $(H,X)$ be a minimal system.
Suppose further that $H$ has a dense subgroup generated by a compact set.
Then $\RP_H^s(X)$ is a closed $H$-invariant equivalence relation, and $(H,X/\RP_H^s(X))$
is the maximal pronilfactor of $(H,X)$ of degree at most $s$.
\end{thm}

The proof of Theorem \ref{th:HKM} (which is a special case of Theorem \ref{th:dynamics}) by Host, Kra and Maass relies
on an ergodic theoretic analogue obtained previously by Host and Kra \cite{HK05}.
Our proof, however, works entirely within the topological category.
We believe that this feature makes our approach of interest even in the case of $\Z$-actions.

We note that in our proof we use the result from \cite{GGY} that $\RP^s_H(X)$
is an equivalence relation for an arbitrary minimal topological dynamical system $(H,X)$.

\subsection{Cubespaces and nilspaces}\label{sc:cubes}

In the preceding discussion we introduced the notion of dynamical cubes.
In this section we discuss similar structures in a much more abstract setting,
following Host and Kra \cite{HK08} and Antol\'\i n Camarena and Szegedy \cite{CS12}.
Everything in this section is taken from \cite{CS12}, although our
terminology and notation differs from that paper.

We first formalize the notion of a cubespace.  Given the previous discussion, the following definitions should seem fairly reasonable; the reader may also consult \cite{GMV1}*{Sections 1--3} for further discussion.

\begin{dfn}
  A map $\f=(\f_1,\ldots,\f_k):\{0,1\}^\ell\to \{0,1\}^k$ is termed a {\bf morphism of discrete cubes }
  if each coordinate function $\f_j(\o_1,\ldots,\o_l)$ is equal to either $0$, $1$, $\o_i$ or $1-\o_i$
  for some $1\le i \le m$.
\end{dfn}

\begin{dfn}
  Let $X$ be a metric space and for each integer $\ell\ge 0$ let $C^\ell\subseteq X^{\{0,1\}^\ell}$
  be a closed set.  We say that $(X,\{C^\ell\}_{\ell\in\N})$ is a {\bf cubespace} if $C^0=X$ and 
  $c\circ\f\in C^\ell$ for any $c \in C^k$ and any morphism of discrete cubes
  $\f:\{0,1\}^\ell \to \{0,1\}^k$.

  We refer to this property as {\bf cube invariance}.
\end{dfn}

We call the elements of $C^\ell$ the $\ell$-{\bf cubes} of $X$.
Given $c \in C^\ell$, we call the elements $c(\o)$ for $\o\in \{0,1\}^\ell$ the {\bf vertices} of $c$.
To simplify our notation, we simply write $X$ to refer to the full cubespace structure $(X, \{C^\ell\}_{\ell\in\N})$, 
and we write $C^\ell(X)$ to refer unambiguously to the cubes $C^\ell$ associated to $X$.

The cube-invariance property encodes certain fairly natural operations that produce new cubes from old ones.  For instance:
\begin{itemize}
  \item cube-invariance for $\f(\o_1,\ldots, \o_{\ell-1})=(\o_1,\ldots, \o_{\ell-1},0)$ encodes the fact that a face of an $\ell$-cube of dimension $(\ell-1)$ is again an $(\ell-1)$-cube;
  \item similarly, cube-invariance for $\f(\o_1,\ldots, \o_{\ell})=(1-\o_1,\ldots, \o_{\ell})$ states that reflecting an $\ell$-cube in one of the coordinate axes yields another $\ell$-cube;
  \item using $f(\o_1, \dots, \o_\ell) = (\o_1, \dots, \o_{\ell-1})$ shows that ``duplicating'' an $(\ell-1)$-cube creates an $\ell$-cube; and
  \item applying $\f(\o_1,\o_2,\o_3)=(\o_2,\o_1,\o_3)$ shows that permuting the coordinates of a cube yields another cube.
\end{itemize}

As well as considering cubespaces, we will also need to discuss maps between them that respect the cubespace structure.  The natural definition is as follows.

\begin{dfn}
  Let $X$ and $Y$ be two cubespaces and let $\f:X\to Y$ be a continuous map.
  Then we say that $\f$ is a {\bf cubespace morphism}, if
  \[
  \{\f\circ c:c\in C^k(X)\}\subseteq C^k(Y)
  \]
  for all $k\in\N$.
\end{dfn}

Cubespaces also admit a natural notion of sub-objects.

\begin{dfn}
  If $X$ and $Y$ are two cubespaces,
  we say that $Y$ is a {\bf subcubespace} of $X$ if $Y\subseteq X$ and $C^\ell(Y)\subseteq C^\ell(X)$ for all $\ell$.
  If $Z\subseteq X$ is a closed set, the subcubespace of $X$ {\bf induced} by  $Z$ is defined by
  $C^\ell(Z)=C^\ell(X)\cap Z^{\{0,1\}^\ell}$ for all $\ell$.
\end{dfn}

We now introduce a further technical definition.
\begin{dfn}
  We say a cubespace $X$ is {\bf ergodic} if $C^1(X)=X^{\{0,1\}}$, i.e.~if any pair of points
  form a $1$-cube. More generally, we say that $X$ is $s$-{\bf ergodic} if $C^s(X)=X^{\{0,1\}^s}$.
\end{dfn}
Observe that $s$-ergodicity implies $t$-ergodicity for all $t\le s$, as a consequence of cube-invariance.

It turns out that non-ergodic cubespaces are fairly uninteresting, insofar as -- under reasonable extra hypotheses -- they decompose into a number of essentially non-interacting ergodic pieces; hence, we will almost always work with ergodic cubespaces as a further sanity condition.

We now turn to the notion of \emph{corner-completion}.  This is both fundamental and rather difficult to motivate from our previous discussion.  Perhaps the best one can say is that this property arises naturally in many of the examples of cubespaces arising in applications -- e.g.~in dynamics, the study of nilmanifolds, etc.~-- and is an essential assumption if one hopes to prove any kind of structure theory.  Ultimately, though, this definition is thoroughly non-obvious, and constitutes one of the key insights of \cite{HK08}.

\begin{dfn}
Let $X$ be a cubespace and let $\l:\{0,1\}^\ell \sm \{\vec1\}\to X$ be a map.
We call $\l$ an $\ell$-{\bf corner} if $\l|_{\o_i=0}$ is an $(\ell-1)$-cube for all $1\le i\le \ell$.
We say that the cubespace $X$ has {\bf $s$-completion}
if any $s$-corner can be completed to an $s$-cube;
in other words, if for any such $\l$ there exists a cube $c\in C^{s}(X)$ such that $c|_{\{0,1\}^{s}\sm \{\vec1\}}=\l$.

We say that $X$ is {\bf fibrant} if it has $s$-completion for all $s\ge0$.
\end{dfn}

It is also imperative to know when this completion process is unique.  The dimension at which this occurs determines the ``degree'' or ``step'' of the space.

\begin{dfn}
  We say that a cubespace $X$ has $s$-{\bf uniqueness},
  if $c_1|_{\{0,1\}^{s}\backslash \vec1}=c_2|_{\{0,1\}^{s}\backslash \vec1}$
  implies $c_1=c_2$ for any two $s$-cubes $c_1,c_2\in C^{s}(X)$.
\end{dfn}

\begin{dfn}
  We say that a cubespace $X$ is a {\bf nilspace} of degree $s$ if it is fibrant and $s\ge 0$ is the
  smallest integer such that $X$ has $(s+1)$-uniqueness.
\end{dfn}

In \cite{GMV1}, the following relative analogue of the corner-completion property is introduced, applying to a map between two cubespaces.

\begin{dfn}
Let $X$ and $Y$ be two cubespaces and let $\f:X\to Y$ be a continuous map.
We say that $\f$ is a {\bf fibration}
if it is a cubespace morphism, and if furthermore the following holds for every $\ell$.

Let $\l: X\to \{0,1\}^\ell \sm \{\vec 1\}$ be an $\ell$-corner and $c\in C^k(Y)$ a compatible cube, in the sense that
$\f\circ\l=c|_{\{0,1\}^\ell\sm\{\vec 1\}}$.
Then there is a completion $c'$ of $\l$ compatible with $c$; that is, there exists $c'\in C^\ell(X)$ such that $c'|_{\{0,1\}^\ell \sm \{\vec1\}} = \l$ and $\f\circ c'=c$.
\end{dfn}

\begin{dfn}
We say that a cubespace morphism $\f:X\to Y$ has $k$-{\bf uniqueness} for some $k\in\N$ if the following holds.
If $c,c'\in\C^k(X)$ are two cubes such that $\f(c)=\f(c')$ and $c(\omega)=c'(\omega)$ for all $\omega\neq\vec 1$,
then, in fact, $c=c'$.

If $\f$ is a fibration that has $k$-uniqueness for some $k\in\N$ and $k$ is the smallest such number, then we say that
$\f$ has {\bf degree} $(k-1)$.
\end{dfn}

It is easy to see that composition of fibrations is a fibration. We also recall
that fibrations have the following ``universal property''.

\begin{lem}[\cite{GMV1}*{Lemma 7.8}]\label{lem:universal1} 
Let $\f_{YX}:X\to Y$ and $\f_{ZX}:X\to Z$ be fibrations between compact cubespaces.
Suppose that for every $y\in Y$ there is $z\in Z$ such that  $\f_{YX}^{-1}(y)\subseteq\f_{ZX}^{-1}(z)$.
Then there is a unique fibration $\f_{ZY}: Y\to Z$ such that $\f_{ZX}=\f_{ZY}\circ\f_{YX}$.

Equivalently, the following holds.
Let $\f:X\to Y$ be a fibration and $\psi: Y\to Z$ be an arbitrary continuous map between two compact cubespaces.
If $\psi\circ\f$ is a fibration then so is $\psi$.
\end{lem}

Observe that a cubespace $X$ is fibrant if and only if the map from $X$ to the one-point cubespace $\{\ast\}$ is a fibration.
If we set $Z = \{\ast\}$ in the above lemma, we see that the image of a fibrant
cubespace under a fibration is also fibrant.

An almost equivalent condition to fibrations appears in \cite{CS12}*{Section 2.8} under the name ``fiber-surjective morphism''.  There, a cubespace morphism $\f \colon X \to Y$ between nilspaces $X$ and $Y$ is called fiber-surjective if the image of any $\sim_k$ class in $X$ is a $\sim_k$ class in $Y$, for any $k \ge 0$.  (Here $\sim_k$ are the canonical equivalence relations, which we will introduce in Section \ref{sc:canonical}.)

If $X$ and $Y$ are nilspaces and $\f \colon X \to Y$ a cubespace morphism, it is not hard to see that $\f$ is a fibration if and only if it is fiber-surjective.  At some points in the project, we have reason to consider fibrations between cubespaces that are not nilspaces, and in these cases the definitions are inequivalent and the notion of a fibration appears to be the correct one to use.  Moreover, in general the authors have also found it a more natural and convenient starting point on which to build the associated theory.

It is proved in \cite{GMV1}*{Remark 7.9} that the image of a nilspace under a fibration is a nilspace. 

Given two $\ell$-configurations $c_{0},c_{1}:\{0,1\}^{\ell}\to X$,
the  {\bf concatenation} of $c_{0}$ and $c_{1}$ is
the $(l+1)$-configuration $[c_{0},c_{1}]:\{0,1\}^{\ell+1}\to X$  given
by $[c_{0},c_{1}](\o,0)=c_{0}(\o)$ and $[c_{0},c_{1}](\o,1)=c_{1}(\o)$
for all $\o\in\{0,1\}^{\ell}$.

\begin{dfn}
  We say that a cubespace $X$ has the {\bf gluing property} if the following holds for all $\ell\ge 0$:
  for all $c_1,c_2,c_3\in C^{\ell}(X)$, $[c_1,c_2],[c_2,c_3]\in C^{\ell+1}(X)$
  implies that $[c_1,c_3]\in C^{\ell+1}(X)$.
\end{dfn}
We recall from \cite{GMV1}*{Proposition 6.2} that fibrant cubespaces always have the gluing property. 
This property is useful for the following reason: if $Y$ is an induced subcubespace of $X$, and $X$ has the gluing property, then so does $Y$; however, if $X$ is fibrant this does not necessarily imply that $Y$ is fibrant.

We note that an ergodic nilspace of degree $0$ is necessarily the $1$ point space.
More interesting examples of cubespaces are the dynamical cubespaces $(X,\{C^k_H(X)\}_{k\in\N})$ introduced 
in Section \ref{sc:non-Ab}.
Indeed, we recall the following fact from \cite{GGY}*{Theorems 3.8 and 7.14}. 
\begin{thm}\label{th:dyn-nilspace}
Let $(H,X)$ be a minimal topological dynamical system.
Then $\RP_H^\ell(X)$ is a closed $H$-invariant equivalence relation,
and $(X/\RP_H^\ell,\{C_H^k\}_{k\in\N})$ is an ergodic nilspace of degree at most $\ell$, for each $\ell\in\N$.  
\end{thm}

Finally, we give another example of a nilspace, which will play a special role in the theory.
Let $A$ be a compact abelian group, with the group operation denoted additively.
We write $\mathcal{D}_s(A)$ for the cubespace defined by requiring that
$c\in C^\ell(\mathcal{D}_s(A))$ if and only if
\[
\sum_{\o\in\{0,1\}^{s+1}}(-1)^{|\o|}c(\f(\o))=0
\]
holds for any morphism of discrete cubes  $\f:\{0,1\}^{s+1}\to \{0,1\}^\ell$,
 where we write $|\o|=\sum_{1\le i\le s+1}\o_i$ for $\o\in\{0,1\}^{s+1}$.
We will consider this object again in the next section, as it turns out to be a
special case of a construction discussed there.
We will see that it follows from general results that $\mathcal{D}^s(A)$ is
a nilspace of degree $s$, and leave it to the reader to verify that it is also $s$-ergodic.

\subsection{Host--Kra nilspaces}\label{sc:HK}

In this section, we discuss a variant of the dynamical cubespace construction considered above.
A more detailed exposition, with examples, is available in \cite{GMV1}*{Section 2 and Appendix A}. 

Let $G$ be a  topological group.
We call a chain of closed subgroups
\[
G=G_0\supseteq G_1\supseteq G_2\supseteq\ldots \supseteq G_{s+1}=\{1\}
\]
a {\bf filtration of degree} $s$ if $[G_i,G_j]\subseteq G_{i+j}$ for all $i,j \ge 0$, adopting the convention that $G_i = \{1\}$ for all $i \ge s+1$.

Note that a filtration is always a central series, but a central series may not be
a filtration.
E.g. if $G$ is a nilpotent Lie group of degree $2$, then $G_0=G_1=G_2=G$, $G_3=[G,G]$,
$G_4=\{1\}$ is a central series, but it is not a filtration, because $[G_2,G_2]\not\subseteq G_4$.
On the other hand, we note that the lower central series is always a filtration (see \cite{MKS66}*{Theorem 5.3}).

We call the filtration {\bf proper} if $G_0=G_1$.
Note that if a group admits a proper filtration of degree $s$ then it must be nilpotent of nilpotency class at most $s$.
In this paper, we always assume that filtrations are proper even if we do not state this explicitly.

We call a group a {\bf filtered} group  if we want to emphasize that it is equipped with a particular
filtration.
We write $G_\bullet$ as a shorthand to denote a group $G$ equipped with a filtration $\{G_i\}$.

We now consider a generalization of the notion of Host--Kra cubegroups introduced in Section \ref{sc:non-Ab}.
A subset $F\subseteq\{0,1\}^{\ell}$ of the vertices of the discrete cube is called a {\bf face
of co-dimension} $d$ if there are indices $1\le i_1<\ldots<i_d\le \ell$ and $\a=(\a_1,\ldots,\a_d)\in\{0,1\}^d$
such that
\[
F=\{\o\in\{0,1\}^\ell:\o_{i_j}=\a_j\;\text{for all $1\le j\le d$}\}.
\]

Let $G_\bullet$ be a filtered topological group of degree $s$.
If $F\subseteq\{0,1\}^\ell$ and $g\in G$, we write $[g]_F$ for the element of $G^{\{0,1\}^\ell}$
given by $[g]_F(\o)=g$ if $\o\in F$ and $[g]_F(\o)=e$ otherwise.

We define the {\bf Host--Kra cubegroup} $\HK^\ell(G_\bullet)$ for each $\ell$ to be the subgroup of $G^{\{0,1\}^{\ell}}$
generated by the elements of the form
$[g]_F$, where $F\subseteq\{0,1\}^\ell$ is a face of codimension $i$ for some $1\le i\le \min(s+1,\ell)$, and $g\in G_i$.

We note that the definition of $\HK^\ell(G_{\bullet})$ in \cite{GMV1}*{Definition 2.2} differs in that 
only ``upper'' faces $F$ are used to construct generators.
However, this gives rise to the same group, as noted in \cite{GMV1}*{Remark 2.3}. 

If $F$ is a face of codimension $d$, then for any positive integers $d_1,d_2$ with $d_1+d_2=d$
we can find faces $F_1$ and $F_2$ of codimension $d_1$ and $d_2$, respectively, such that
$F_1\cap F_2=F$.
Then, the identity $[[g_1]_{F_1},[g_2]_{F_2}]=[[g_1,g_2]]_{F}$ holds (where, confusingly, some of the square brackets denote commutators and others do not).
Using these observations, it is easy to see that the Host--Kra cubegroup $\HK^\ell(G)$ defined in Section
\ref{sc:non-Ab} agrees with the above construction applied to the lower central series filtration.
For a different filtration, however,
the Host--Kra cubegroup may be larger.

We digress to justify a claim we made in Section \ref{sc:non-Ab}.
Let $g\in G_{s+1}$, the $(s+1)$-th element of the lower central series filtration.
Then $[g]_F\in \HK^{s+1}(G)$ for any face $F$ of co-dimension $(s+1)$ in $\{0,1\}^{s+1}$,
i.e.~for any single vertex of $\{0,1\}^{s+1}$.
Thus, $[g]_F.x^{\{0,1\}^{s+1}}\in C^{s+1}_G(X)$,
which implies that $(x,gx)\in \RP_G^{s+1}(X)$ as claimed in Section \ref{sc:non-Ab}.

Let $G_\bullet$ be a degree $s$ filtered Lie group.\footnote{Note that, by Cartan's closed subgroup theorem, the groups $G_i$ appearing in the filtration of a Lie group are automatically themselves Lie groups.}
We say that a discrete co-compact subgroup $\Gamma$ of $G$ is {\bf compatible}
with the filtration if $\Gamma\cap G_i$
is a (discrete) co-compact subgroup of $G_i$ for all $i$.

The group $\HK^\ell(G_\bullet)$ naturally acts on the space
$(G/\Gamma)^{\{0,1\}^\ell}$.
It turns out that
the stabilizers are discrete cocompact subgroups in $\HK^\ell(G_\bullet)$,
provided $\Gamma$ is compatible with the filtration, and hence the orbits of the action
are compact and therefore closed.
(For a proof of this fact, see \cite{GT10}*{Lemma E.10}, where the only fact that is used (implicitly) is that $\Gamma$ is compatible with the filtration.)

We define the {\bf Host--Kra nilspace} $\HK(G_\bullet)/\Gamma$ associated to $G_\bullet$ and $\Gamma$ as follows.
The base topological space is $X = G/\Gamma$, and the cubes are defined as
\[
  C^\ell(G/\Gamma) := \left\{\omega \mapsto g(\omega).x \colon g\in \HK^\ell(G_\bullet), x\in X \right\}.
\]

In \cite{GMV1} we defined Host--Kra nilspaces slightly differently:
we considered $(G,\{\HK^\ell(G_\bullet)\}_{\ell\in\N})$ as a cubespace, and defined the cubespace $G/\Gamma$ 
as the quotient of this under the map $G\to G/\Gamma$; so the cubes of $G/\Gamma$ are denoted as $\HK^\ell(G_\bullet)/\Gamma$.
However, this is completely equivalent to the above definition, and the discrepancy made deliberately for consistency
with the dynamical viewpoint adopted in this paper.

It is proven in \cite{GMV1}*{Proposition 2.6} 
that $\HK(G_\bullet)/\Gamma$ is an ergodic nilspace of degree $s$.

Recall the definition of $\mathcal{D}_s(A)$ from the previous section.
It turns out that $C^\ell(\mathcal{D}_s(A))=\HK^\ell(A_\bullet)$, where
$A$ is considered with the filtration $A_0=A_1=\ldots = A_s=A$ and $A_{s+1}=\{0\}$
of degree $s$.
This equivalence is proved in \cite{GMV1}*{Propostion 5.1}. 
When $A$ is a compact Lie group, this is a Host--Kra nilspace.

Host--Kra nilspaces are relatively easy to understand thanks to the well-developed
theory of nilmanifolds.
The main aim of our project following Antol\'\i n Camarena and Szegedy \cite{CS12}
 is that we want to approximate  general nilspaces by Host--Kra nilspaces.
We outline our program to achieve this goal in the next sections.

\subsection{Canonical factors}\label{sc:canonical}

The first stage of our program (following \cite{CS12}) is to realize an ergodic nilspace of degree $t$ as a tower of extensions
\[
X\to\pi_{t-1}(X)\to\pi_{t-2}(X)\to\ldots\to\pi_0(X)=\{\ast\},
\]
where the degree of the nilspace is reduced by one each time we move along the sequence.

In the setting of Host--Kra nilspaces, this corresponds to taking quotients by the normal subgroups
$G_s$ that form the filtration; i.e.~we expect that $\pi_s(G/\Gamma)$ should be $\HK((G_\bullet/G_{s+1}) / (\Gamma / (\Gamma \cap G_{s+1})))$.  The challenge is to simulate this construction in the setting of general nilspaces.

It is clear that we will need to construct $\pi_s(X)$ as quotients of $X$ in a suitable sense, and we first verify that this makes sense.

\begin{dfn}
  Let $X$ be a cubespace and let $\sim$ be a closed equivalence relation on $X$.
  Write $\pi:X\to X/\sim$ for the corresponding quotient map.

  Then we define a cubespace structure on $X/\sim$, the {\bf quotient cubespace}, by declaring a configuration $c\in (X/\sim)^{\{0,1\}^\ell}$
  to be a cube if and only if there is a cube $c'\in C^{\ell}(X)$ such that $\pi(c')=c$.
\end{dfn}
It is easy to verify that $X/\sim = \pi(X)$ as constructed is indeed a cubespace.

Resuming the above discussion: the key feature of $\pi_s(G/\Gamma)$ for a Host--Kra nilspace is that it has degree $s$, and is in some sense the largest quotient with this property.  In cubespace language, this states that $\pi_s(G/\Gamma)$ has $(s+1)$-uniqueness.

So, for $X$ a general cubespace,
our task is to find an equivalence relation $\sim_s$ on $X$ such that the quotient $X/\sim_s$ has $(s+1)$-uniqueness,
and it is the smallest relation with this property.  The correct definition is as follows.

\begin{dfn}
  Given a cubespace $X$ and $s \ge 0$, define a relation $\sim_s$ on $X$ as follows: $x\sim_i y$ if and only if there are two cubes $c_1,c_2\in C^{i+1}(X)$
  such that $c_1(\o)=c_2(\o)$ for $\o\neq\vec 1$, $c_1(\vec 1)=x$ and $c_2(\vec 1)=y$.

  We call $\sim_s$ it the $s$-th {\bf canonical equivalence relation} on $X$.
\end{dfn}
This may be compared with \cite{CS12}*{Definition 2.3}.

It is clear that if $X/\sim_s$ is to have $(s+1)$-uniqueness, then $\sim_s$ must contain at least these pairs of points.
What is much less obvious is that this definition does indeed give rise to a closed equivalence relation, and that the corresponding quotient is a nilspace whenever $X$ is.
However, all this is proved in \cite{GMV1}*{Proposition 6.3}, 
following \cite{CS12}*{Section 2.4} and \cite{HK08}*{Section 3.3} closely.

Note that $\sim_s$ is the trivial relation if and only if $X$ already has $(s+1)$-uniqueness.

The canonical equivalence relation $\sim_s$ has the following alternative definition, whose equivalence is
proven in \cite{GMV1}*{Lemma 6.6} 
(following \cite{CS12}*{Lemma 2.3} and \cite{HK08}*{Proposition 3}).
Recall from Section \ref{sc:non-Ab} that we denote by $\llcorner^s(x;y)$ the configuration $\o\mapsto x$
for $\o\neq\vec1$ and $\vec1\mapsto y$.

\begin{lem}\label{lem:alter-canonical}
  Let $X$ be a fibrant cubespace.
  Then $x\sim_s y$ if and only if $\llcorner^{s+1}(x;y)$ is a cube.
\end{lem}

This alternative characterization of $\sim_s$ shows that
if $X$ is a fibrant dynamical cubespace induced by a group $G$ acting on $X$, then $\sim_s=\RP_G^s$.

Now we can explain why $\RP_G^s$ is trivial on a degree $s$ nilsystem, as we claimed
at the end of Section \ref{sc:RP} for $\mathbb{Z}$-systems and in Section \ref{sc:non-Ab} for general $G$-actions.
In this case, the dynamical cubes form a Host--Kra cubespace in the above sense, which is a
nilspace of degree at most $s$ by \cite{GMV1}*{Proposition 2.6}. 
Hence, $\sim_s$ is trivial (by the above remarks) and it follows that $\RP_G^s$ is, since they agree.

The canonical equivalence relation has the following {\bf universal replacement}
property proved in \cite{GMV1}*{Proposition 6.3} 
(see also \cite{CS12}*{Lemma 2.5} and \cite{HK08}*{Proposition 3}).
\begin{lem} \label{lem:replacement}
  Let $X$ be a fibrant cubespace.
  Let $c\in C^{\ell}(X)$ for some $\ell\le s+1$, and
  suppose that $c'\in X^{\{0,1\}^\ell}$ is a configuration such that $c(\o)\sim_s c'(\o)$
  for all $\o\in\{0,1\}^\ell$.
  Then $c'\in C^\ell(X)$.
\end{lem}

Finally, we summarize our notation for these constructions.
If $X$ is a fibrant cubespace,
we call $X/\sim_s$ the $s$-th {\bf canonical factor} of $X$.
The quotient map is denoted $\pi_s \colon X \to X/\sim_s$, and we also use the notation $\pi_s(X)$ to denote the quotient space.

\subsection{Structure groups}\label{sc:structure-groups}

We state the first structure theorem for nilspaces in this section, which
is proved in  \cite{GMV1}*{Theorem 5.4} 
(see also \cite{CS12}*{Theorem 1}, and also \cite{HK08}*{Section 5} for a closely related discussion).

\begin{thm}[Weak Structure Theorem]\label{th:weak-structure}
Let $X$ be a compact ergodic nilspace of degree $s$.
Then there is a compact abelian group $A = A_s(X)$ (notated additively) acting continuously and freely on $X$, such that the following hold.
\begin{enumerate}
\item The orbits of $A$ coincide with the fibres of $\pi_{s-1}$, the $(s-1)$-th canonical projection.
\item Let $c_1\in C^{\ell}(X)$ and $c_2:\{0,1\}^\ell\to X$ be such that $\pi_{s-1}(c_1)=\pi_{s-1}(c_2)$.
  Denote by $a:\{0,1\}^{\ell}\to A$ the unique configuration in $A$ such that $a(\o).c_1(\o)=c_2(\o)$ for all $\o \in \{0,1\}^\ell$.
Then $c_2\in C^\ell(X)$ if and only if $a\in C^{\ell}(\mathcal{D}_s(A))$.
\end{enumerate}
\end{thm}

It is worth spelling out the meaning of item $2$ in a few special cases.
If $\ell=s+1$, then the condition $a\in C^{\ell}(\mathcal{D}_s(A))$ is equivalent to
\[
\sum_{\o\in\{0,1\}^{s+1}}(-1)^{|\o|}a_\o=0.
\]

If $l<s+1$, then any configuration $c:\{0,1\}^\ell\to X$ is a cube provided $\pi(c)\in C^\ell(\pi(X))$, because
$\mathcal{D}_s(A)$ is $s$-ergodic.  Note this is consistent with Lemma \ref{lem:replacement}.

If we consider cubes contained in a single fibre of $\pi_{s-1}$, we see that they admit a free and transitive action by $C^\ell(\mathcal{D}_s(A))$.  Equivalently, the subcubspaces of $X$ induced by the fibres of $\pi_{s-1}$ are all isomorphic copies of $\mathcal{D}_s(A)$.

Using the characterization of $\mathcal{D}_s(A)$ in terms of Host--Kra cubegroups, we obtain the
following.
If $c\in C^{\ell}(X)$ is a cube, and $F$ is a face of codimension at most $s$, then
$[a]_F.c$ is also a cube for all $a\in A$.
This holds in particular when $\ell=s+1$ and $F$ is an edge, i.e.~a face of dimension $1$.

The weak structure theory constitutes progress towards a structure theorem for nilspaces, for the following reasons.
\begin{itemize}
  \item By part (1), we have that $X$ is an $A$-principal bundle over $\pi_{s-1}(X)$.  Iterating this procedure on $\pi_{s-1}(X)$, we realize $X$ as a tower of extensions by compact abelian groups, terminating in $\pi_0(X)$, the $1$-point space, which constitutes fairly strong information about the structure of $X$.
  \item In order to recover the cubes of $X$ given knowledge of $\pi_{s-1}(X)$, it suffices to exhibit a single $(s+1)$-cube lying above each $(s+1)$-cube of $\pi_{s-1}(X)$, since part (2) then gives us all such cubes.  Again, we can iterate this on $\pi_{s-1}(X)$ to obtain a full description of the cubespace.
\end{itemize}
We call the group $A = A_s(X)$ the {\bf top structure group} of $X$. Also, we define $A_t(X) := A_t(\pi_t(X))$, the top structure group of the canonical factor $\pi_{t}(X)$ for $0 \le t \le s$, and call it the $t${\bf -th structure group} of $X$.

The proof of Theorem \ref{th:weak-structure} is given in the paper \cite{GMV1}, we only recall here how the group $A$ is constructed.
Recall that we denote by $\llcorner^\ell(x;y)$ the $\ell$-configuration all of whose vertices are $x$ except for the one labelled
by $\vec 1$, which is $y$.
Recall also the notation $[c_1,c_2]$ from Section \ref{sc:cubes},
which denotes the concatenation of the configurations  $c_1$ and $c_2$.

We consider the set
\[
Y=\{(x,y)\in X\times X:x\sim_{s-1} y\}
\]
and introduce a relation $\approx$ on $Y$, given by
setting $(x,y)\approx(x',y')$ if and only if $[\llcorner^{s}(x;y),\llcorner^{s}(x';y')]\in C^{s+1}(X)$.

It is shown in \cite{GMV1} that $\approx$ is a closed equivalence relation; hence we define
$A$ to be the quotient $Y/\approx$.
One can argue using $(s+1)$-uniqueness that, given $x \in X$, each equivalence class of $\approx$ has
an unique representative of the form $(x,y)$ for some $y$.
Hence each element of $A$ can be identified with the graph of a transformation on $X$, and we
use this identification to define simultaneously the group law on $A$ and its action on $X$.

\subsection{Lie-fibered nilspaces}

We say that a nilspace $X$ is {\bf Lie-fibered} if the structure groups $A_i(X)$ of $X$ defined in
the previous section are all Lie groups.
We recall the main result of the paper \cite{GMV2} below, which classifies Lie-fibered nilspaces under the additional technical assumption that  $C^k(X)$
is connected for each $k$.
We call nilspaces satisfying this latter property {\bf strongly connected}.

We say that a homeomorphism $f$ of $X$ is an $i$-{\bf translation} if
$[f]_F.c\in C^{\ell}(X)$ for each $c\in C^\ell(X)$
and any face $F\subseteq \{0,1\}^\ell$ of the discrete cube of codimension $i$.  It follows in particular that $f$ is a cubespace morphism $X \to X$.
It is clear that the set of $i$-translations, endowed with the maximum displacement
metric, forms a topological group, which we denote by $\Aut_i(X)$.
The notion of translations originate from the work of Host and Kra \cite{HK08}*{Definition 6}
and they play a prominent role in the program of
Antol\'\i n Camarena and Szegedy \cite{CS12}.

It is easy to verify from the definitions that the groups $\Aut_i(X)$ are nested, and form a (proper) filtration of the group $\Aut_1(X)$, by
\[
  \Aut_1(X) \supseteq \Aut_1(X) \supseteq \Aut_2(X) \supseteq \dots  \ .
\]
The reason for $\Aut_1(X)$ appearing twice in the above line is that $\Aut_1(X)$ is both the index $0$ and the index $1$ element of the filtration.
We denote this filtered group by $\Aut_\bullet(X)$.

Let $G_\bullet$ be a filtered Lie group with a compatible discrete cocompact subgroup $\Gamma$.
It is a direct consequence of the definitions that
the elements of $G_i$ are $i$-translations on the Host--Kra nilspace $\HK(G_\bullet)/\Gamma$.
Therefore, if it is possible to represent a nilspace $X$ as a Host--Kra nilspace, then it must be possible to
locate the filtered group $G_\bullet$ inside the filtered group $\Aut_\bullet(X)$.
The next result confirms that, in the case of Lie-fibered strongly connected nilspaces,
it is sufficient simply to take $G_i = \Aut_i^\circ(X)$, the connected component of the identity in $\Aut_i(X)$.
For a proof, see \cite{GMV2}*{Theorem 2.18} 
(and see also \cite{CS12}*{Theorem 7}).

\begin{thm}\label{th:toral}
Let $X$ be a compact ergodic Lie-fibered strongly connected nilspace of degree $s$.
Fix a point $x_0\in X$.
Then $G=\Aut^\circ_1(X)$ is a Lie group which admits a filtration
\[
G_\bullet:\,G=G_0=\Aut^\circ_1(X)\supseteq \Aut^\circ_2(X)\supseteq\ldots \supseteq \Aut^\circ_{s+1}(X)=\{1\}
\]
of degree $s$ and a discrete subgroup
$\Gamma=\Stab(x_0)\subseteq G$ compatible with the filtration, such that the map
\begin{align*}
  G/\Gamma &\to X \\
  f\cdot \Ga &\mapsto f(x_0)
\end{align*}
is an isomorphism of cubespaces between $\HK(G_\bullet)/\Gamma$
and $X$.
\end{thm}

In \cite{CS12}*{Theorem 7} the same conclusion is shown to hold
under the assumption that the structure groups of $X$ are all connected, i.e.~are all tori
of various dimensions.
It is easy to see that this condition implies that $X$ is strongly connected.
The other implication -- that strong connectivity implies that the structure groups are tori --
also holds, but it is less obvious.
Indeed, the only proof of which we are aware makes use of the full force of our
structure theorem:
Theorem \ref{th:toral} implies that a strongly connected nilspace with Lie structure
groups is isomorphic to a Host--Kra nilspace of a connected nilpotent Lie group $G$ endowed with a filtration
$\{G_i\}$ of connected subgroups, and it follows from this that the structure groups $(G_i/G_{i+1}) / ((G_i \cap \Gamma) / (G_{i+1} \cap \Gamma))$ are tori.

There is a third possible formulation of Theorem \ref{th:toral}.
One can replace the condition that $X$ is Lie-fibered with suitable topological conditions, e.g.~requiring
that $X$ is locally connected and has finite Lebesgue covering dimension.  These conditions certainly hold whenever $X$ is a topological manifold, which is clearly a necessary condition for $X$ to be isomorphic to a nilmanifold.  See Theorem \ref{th:top-conditions} for further details.

In the setting of nilspaces constructed from a topological dynamical system, as in Section \ref{sc:non-Ab},
we have the following variant.
Recall that a nilsystem is a topological dynamical system $(H,X)$ such that there is a nilpotent
Lie group $G$ and a discrete cocompact subgroup
$\Gamma$ of $G$, such that $(H,X)$ is isomorphic to $(H,G/\Gamma)$, where the action of $H$
on $G/\Gamma$ is induced from a continuous homomorphism $\alpha \colon H\to G$.

\begin{thm}\label{th:Lie-dynamical}
  Let $(H,X)$ be a minimal topological dynamical system.
  Assume that $X$ is locally connected and has finite Lebesgue covering dimension, and also that $\RP_H^s(X)$, the regional proximal relation of order $s$, is trivial for some $s$.

  Then $(H,X)$ is a nilsystem.
\end{thm}

The proof of this variant of Theorem \ref{th:toral} is discussed in
Section \ref{sc:proof-dyn-alg-struc} of the appendix. A
similar result is proved in \cite{GMV2}*{Corollary 2.20}, 
where as in Theorem \ref{th:toral} the topological conditions are
replaced by cubespace-theoretic conditions on the dynamical nilspace constructed from the action of $H$.

We recall the main ideas of the proofs of these  theorems from the paper \cite{GMV2}.
The most difficult part of the proof of Theorem \ref{th:toral}
 is in verifying that the group $\Aut_1^\circ(X)$ acts transitively
on $X$. This is proved by induction on the degree of $X$.
It can be seen easily that the canonical projection $\pi_{s-1}$ induces a homomorphism
$\pi_{s-1}^*:\Aut_1^\circ(X)\to\Aut_1^\circ(\pi_{s-1}(X))$.
Furthermore, it can be seen that the kernel of $\pi_{s-1}^*$ contains the connected
component of the top structure group $A$ of $X$,
which acts transitively on the connected components of the fibres of
$\pi_{s-1}$ as we discussed in the previous section.
Therefore, it remains to prove that $\pi_{s-1}^*$ is surjective, from which transitivity follows by these observations and inductive hypothesis on $\pi_{s-1}(X)$.

Fix a small parameter $\ve>0$ and let $f\in\Aut_1^\circ(\pi_{s-1}(X))$ be a translation such that $\dist(x,f(x))<\ve$
for all $x\in\pi_{s-1}(X)$.
We want to show that there is a translation $\wt f\in\Aut_1^\circ(X)$ such that $\pi_{s-1}^*(\wt f)=f$.

To this end, we first find a homeomorphism $g$ of $X$ -- not necessarily a translation or even a cubespace morphism --
such that $\pi_{s-1}(g(x))=f(\pi_{s-1}(x))$; and moreover such that $g$ commutes with the action of the top structure group $A_s(X)$, i.e.~$g(a.x)=a.g(x)$ for all $a\in A$ and $x\in X$.

In the next step, we attempt to correct $g$ so as to make it a genuine translation on $X$.  Specifically, we look for a map $\a:X\to A$ such that the transformation
\[
\wt f: x\mapsto \a(x).g(x)
\]
is a translation.

For $c\in C^{s}(X)$, we write $D(c)$ for the unique element of $A$ such that
$[c,[D(c)]_{\{\vec 0\}}.g\circ c]\in C^{s+1}(X)$.
The function $c\mapsto D(c)$ encodes the amount $g$ deviates from being a translation at $c$.
It can be seen from the weak structure theorem that the transformation
$\wt f$ defined above is a translation if and only if the functional equation
\begin{equation}\label{eq:funct}
\sum_{\o\in\{0,1\}^s} (-1)^{|\o|}\a(c(\o))=D(c)
\end{equation}
holds for all $c \in C^s(X)$.

The equation \eqref{eq:funct} appears in \cite{CS12} (although not in exactly this way) and plays a prominent role in the whole theory.
We recall some definitions and then a result about the solutions of \eqref{eq:funct} from \cite{GMV2}.

\begin{dfn}
Let $A$ be a compact abelian Lie group, let $X$ be a cubespace, and let $\ell \in \N$ be given.
A continuous function $\sigma:C^{\ell}(X)\to A$
is called an $\ell$-{\bf cocycle} (or just a cocyle) if
\[
  \rho([c_1,c_3])=\rho([c_1,c_2])+\rho([c_2,c_3])
\]
holds for any $c_{1},c_{2},c_3\in C^{\ell-1}(X)$ such that all three concatenations appearing in the equation are cubes.

We call this property \emph{additivity}.
\end{dfn}
We stress that our rather vague notation allows the concatenation operation $[-,-]$
on any coordinate $\{1,\ldots,\ell\}$, not just the first one; hence there are strictly speaking $\ell$ additivity conditions, one per coordinate.

We note the following two consequences of additivity:
\begin{itemize}
\item (degenerate cubes) if $c=[c_0,c_0]$ then $\rho(c)=0$;
\item (reflections) we have $\rho([c_0,c_1])=-\rho([c_1,c_0])$.
\end{itemize}

Again let $X$ be a cubespace, $A$ an abelian group and $\ell\in\N$,
and let $f:X\to A$ be a function.
Then we define the function $\partial^\ell f:C^\ell(X)\to A$ by
\[
\partial^\ell f(c)=\sum_{\o\in\{0,1\}^\ell} (-1)^{|\o|}f(c(\o)).
\]

The key step in the proof of Theorem \ref{th:toral} is the following result, which
guarantees the existence of solutions to \eqref{eq:funct} under certain hypotheses.
In fact, we state the theorem in a slightly more general form than is necessary for the purposes of
Theorem \ref{th:toral}, but we will need the full power of it in the proof of the results
that we state in the next section.

\begin{thm}
\label{th:functional}

  Let $A$ be a compact abelian Lie group and let $s \ge 0$, $\ell \ge 1$ be given.  Then there exists $\delta = \delta(s, \ell, A) > 0$ such that the following holds.

  Let $\f \colon X \to Y$ be any fibration of degree $s$ between compact ergodic cubespaces $X$ and $Y$ that obey the gluing axiom. Let $\rho$ be an  $\ell$-cocycle on $X$ with values in $A$, let $0 < \delta' \le \delta$ be given and suppose that $\dist(\rho(c), \rho(c')) \le \delta'$ whenever $\f(c) = \f(c')$.

Then there is a continuous map $f:X\to A$ and a cocycle $\wt\rho: C^\ell(Y)\to A$ such that
\begin{equation}\label{eq:functional}
\rho=\partial^\ell f+(\wt\rho\circ\f)
\end{equation}
  and $\dist(f(x), f(y)) \lesssim_{s,\ell} \delta'$ (that is, there exists a constant $c=c(s,\ell)>0$ such that $\dist(f(x), f(y))\leq c\delta'$) whenever $\f(x) = \f(y)$.

  Moreover, the function $f$ is unique in the following sense. Suppose $f,f'$ are two continuous solutions of \eqref{eq:functional}; in particular, $f$ and $f'$ are continuous functions such $\partial^\ell(f - f')$ is constant on fibers of $\f$. Suppose moreover that $\dist(f(x),f'(x))\le\d$
for all $x\in X$. Then $f-f'$ is constant on the fibres of $\f$.
\end{thm}

This result in this form is proved in \cite{GMV2}*{Theorem 5.1 and Corollary 5.3}, 
but it is very closely modelled on \cite{CS12}*{Lemma 3.19}.
The main difference is that \cite{CS12} considers only the special case where $Y$ is the one point
space (which is all that is required in our proof of Theorem \ref{th:toral}).
The general case requires, among other things, the ``relative'' weak structure theory developed in \cite{GMV1}.

The proof of Theorem \ref{th:toral} uses the connectedness hypothesis  in two
crucial ways.
First, we can find continuous lifts only for translations of $\pi_{s-1}(X)$ with small
displacement.
Second, we can solve the functional equation \eqref{eq:funct} only for small cocycles.
In the proof of Theorem \ref{th:Lie-dynamical} we can get around this problem using that the acting group $H$
immerses into the group of translations and it acts transitively on the space of connected components.
We can then realize $X$ as a homogeneous space of $G=\langle \Aut_1^\circ(X), H\rangle$.
Observe that $G$  need not be connected
even if $X$ is.

The approach in \cite{CS12} to the proof of Theorem \ref{th:toral}
is both closely related to ours and in other ways somewhat different.
Both proofs use at their core the triviality of certain cocycles, in the sense of Theorem \ref{th:functional},
but the way these arise, and the method of constructing small translations, vary.

The approach in \cite{CS12} can be summarized as follows.
The authors develop a kind of cohomology theory for nilspaces, whereby an extension of a nilspace by an abelian group may be characterized
up to isomorphism by a measurable cocycle (up to ``coboundaries'').
In this picture, cocycles which are ``trivial'', or equal to coboundaries (i.e.~of the from $\rho = \partial^\ell g$) are shown to correspond to split or direct product extensions.
Armed with these tools, and given an element
$f\in\Aut_1^\circ(\pi_{s-1}(X))$ with small displacement, they construct an extension of a certain nilspace by a compact abelian group, such that the extension splits if and only if $f$ has a lift in $\Aut_1^\circ(X)$.  Hence, the problem is reduced to showing triviality of a measurable cocycle associated to this extension.

\subsection{Inverse limits}
\label{subsec:invlim}
We turn to the structure theory of (general) nilspaces.
It turns out that these are not all Lie-fibered nilspaces (as can be seen by considering $\mathcal{D}_s(A)$ where $A$ is a compact abelian group but not a Lie group); but they can be approximated
by Lie-fibered nilspaces in some sense, as the following result (identical to \cite{CS12}*{Theorem 4}) shows.

\begin{thm}[Inverse Limit Theorem]\label{th:invlim}
  Let $X$ be a compact ergodic nilspace.
  Then there is a sequence of compact ergodic Lie-fibered nilspaces $\{X_n\}$,
  and an inverse system of fibrations $\{\f_{m,n}:X_n\to X_m\}_{m<n}$
  such that $X=\invlim X_n$.
\end{thm}

The proof of this result is given in Section \ref{sec:cubeinvlim}; we now outline the main ideas.

The first step is to note that compact abelian groups are inverse limits of
compact abelian Lie groups, and to apply this fact to the structure groups of $X$.
This is enough to deduce the degree $1$ case; we prove
the theorem for higher degree nilspaces by an inductive argument.

We may identify the top structure group of $X$ with an inverse limit of Lie groups,
and thereby write $X=\invlim X_\infty^{(m)}$
where each $X_\infty^{(m)}$ is a quotient of $X$ under the action of a subgroup
of the top structure group, and the top structure group of $X_\infty^{(m)}$ is a Lie group.
The degree $(s-1)$ factors are unaffected, i.e.~$\pi_{s-1}(X)=\pi_{s-1}(X_\infty^{(m)})$.
Furthermore, we may construct the sequence so that we eventually quotient by the whole top structure group, i.e.~$X_\infty^{(0)}=\pi_{s-1}(X)$.

Next, we apply the induction hypothesis to $X_\infty^{(0)}$ to write
$X_\infty^{(0)}=\invlim X_n^{(0)}$, where $X_n^{(0)}$ is Lie-fibered.

So far, we have some degree $(s-1)$ Lie-fibered nilspaces $X_n^{(0)}$ approximating $\pi_{s-1}(X)$, and some spaces $X_\infty^{(m)}$ approximating $X$ whose top structure groups are Lie, but whose degree $(s-1)$ quotients are still huge.

Our remaining task is to fill in the diagonal: we want to build a space $X_n^{(m)}$ for enough pairs $n$ and $m$, whose degree $(s-1)$ factor is the Lie-fibered space $X_n^{(0)}$ and whose top structure group is the same as that of $X_\infty^{(m)}$.

More precisely, we want to construct a nilspace $X_n^{(m)}$ for each $m$ and all sufficiently large $n \ge n_0(m)$ (depending on $m$),
such that $\pi_{s-1}(X_n^{(m)})=X_n^{(0)}$ and the fibration
$X_\infty^{(0)}\to X_n^{(0)}$ can be lifted to a fibration $X_\infty^{(m)}\to X_n^{(m)}$.
Once this is done, we can write $X_\infty^{(\infty)}$ as the inverse limit of a sequence  of the form $\{X_{n_m}^{(m)}\}_{m\in\N}$ for some sequence $n_m \rightarrow \infty$.

In other words, we wish to construct the following commuting diagram of fibrations, and then take an inverse limit up the diagonal.

\begin{tikzpicture}
\matrix(m) [matrix of math nodes, row sep=3em, column sep=3em]
{
 X_\infty^{(\infty)}=X & & & \\
 X_\infty^{(m+1)} &  X_{n_{m+1}}^{(m+1)} & & \\
 X_\infty^{(m)} &  X_{n_{m+1}}^{(m)} &  X_{n_{m}}^{(m)} & \\
 X_\infty^{(0)}=\pi_{s-1}(X) &  X_{n_{m+1}}^{(0)}
&  X_{n_{m}}^{(0)} &  X_0^{(0)}=\{\bullet\}\\};
\path[-stealth]
(m-1-1) edge  (m-2-1)edge [dashed] (m-2-2)
(m-2-1) edge  (m-3-1)edge  (m-2-2)
(m-3-1) edge  (m-4-1)edge  (m-3-2)
(m-4-1) edge  (m-4-2)
(m-2-2)  edge  (m-3-2)edge [dashed] (m-3-3)
(m-3-2) edge  (m-4-2) edge  (m-3-3)
(m-4-2) edge  (m-4-3)
(m-3-3) edge  (m-4-3)edge [dashed] (m-4-4)
(m-4-3) edge  (m-4-4)
;
\end{tikzpicture}

The main difficulty of the approach lies in the construction of this nilspace $X_n^{(m)}$.  To reiterate this isolated problem: we are given a fibration $X_\infty^{(0)} \to X_n^{(0)}$, and also that the nilspace $X_\infty^{(0)}$ is the quotient of $X_\infty^{(m)}$ under the free action of a compact abelian group $A = A_s(X_\infty^{(m)})$. We wish to construct $X_n^{(m)}$ and a fibration $X_\infty^{(m)} \to X_n^{(m)}$, such that $X_n^{(0)}$ is the
quotient of $X_n^{(m)}$ under the free action of the same group $A$, and such that the following diagram commutes:

\begin{center}
\begin{tikzpicture}
\matrix(m) [matrix of math nodes, row sep=3em, column sep=3em]
{
 X_\infty^{(m)} &  X_{n}^{(m)} & & \\
 X_\infty^{(0)} &  X_{n}^{(0)} & & \\
};
\path[-stealth]
(m-1-1) edge (m-1-2) edge node [left] {$\pi_{s-1}$} (m-2-1)
(m-1-2) edge node [right] {$\pi_{s-1}$} (m-2-2)	
(m-2-1) edge  (m-2-2)	
;
\end{tikzpicture}
\end{center}

It turns out that this is not a natural or categorical construction, and indeed it is not possible in general to construct $X_n^{(m)}$ with these properties.  To do so, we will need some topological input; in particular, we will have to make use of the fact that $A$ is Lie and that $X_n^{(0)}$ is ``sufficiently close'' to $X_\infty^{(0)}$ if $n$ is large enough, in the sense that the fibers of the map $X_\infty^{(0)} \to X_n^{(0)}$ have small diameter.

The (only reasonable) way to construct the nilspace $X_n^{(m)}$  is as a quotient of $X_\infty^{(m)}$ by a
closed equivalence relation, which we denote by $\sim_n^m$.
For the quotient to have the required properties, we need the following hold for every equivalence class $D\subseteq X_\infty^{(m)}$ of $\sim_n^m$:
\begin{itemize}
  \item the image $\pi_{s-1}(D) \subseteq X_\infty^{(0)}$ is equal to the inverse image of a single point under the fibration $X_\infty^{(0)}\to X_n^{(0)}$; and
  \item the restriction
    \[
      \pi_{s-1}|_D \colon D \to \pi_{s-1}(D)
    \]
    of the canonical projection $\pi_{s-1}$ to the subcubespace induced by $D$, is a cubespace isomorphism.
\end{itemize}
We call a set $D$ satisfying these two properties a {\bf straight class}.

So, to construct $\sim_n^m$, we need to show that if $n$ is sufficiently large depending on $m$, then
given any point $x \in X_\infty^{(m)}$ we can find a canonical straight class containing $x$.

The proof of this fact proceeds as follows.  First, we need to invoke with Gleason's theorem on the existence of local sections
for the bundle $X_\infty^{(m)}\to X_\infty^{(0)}$.
I.e.~this states that any point $x\in X_\infty^{(0)}$ has a neighborhood $U$ that admits a continuous
map $\s:U\to X_\infty^{(m)}$ such that $\pi_{s-1}\circ \s=\Id_U$.

Next, we want to adjust this section $\s$ so that it sends fibers of $X_\infty^{(0)} \to X_n^{(0)}$ to straight classes in $X_\infty^{(m)}$.

In other words, we need to choose a suitable map
$f:U\to A$ (where as above $A = A_s(X_\infty^{(m)})$ is the top structure group of $X_\infty^{(m)}$) and set $\s'(x)=f(x).\s(x)$ for $x\in U$.
The condition on $f$ that asserts that $\s'$ maps inverse images of points in $X_n^{(0)}$ into
straight classes, is very similar to \eqref{eq:funct}, and so we are able to find
such an $f$ using Theorem \ref{th:functional}.

The approach of \cite{CS12} to proving the inverse limit theorem is different.
As we mentioned previously, Antol\'\i n Camarena and Szegedy develop a cohomology theory for nilspaces.
By an intricate argument measurable cocycles are shown to correspond to extensions by compact groups.
This principle is then shown to be valid in a relative setting.
Indeed, the nilspace  $X_{n}^{(m)}$ is constructed from a measurable
cocycle arising from a section for $X_{\infty}^{(m)}\to X_{\infty}^{(0)}$
which is compatible with $X_{n}^{(0)}$.

\subsection{Equivariance of fibrations for translation groups}

We have seen in the previous section that strongly connected Lie-fibered nilspaces
can be endowed with the structure of a Host--Kra nilspace.
It is natural to ask whether this structure is respected by the maps in the inverse system
realizing a strongly connected nilspace as the inverse limit of Host--Kra nilspaces.
The answer turns out to be positive, as confirmed by the following theorem that
will be obtained in Section \ref{sc:functor} as a byproduct of the proof of Theorem \ref{th:invlim}.
The results stated in this section are new.

\begin{thm}\label{th:functoriality}
Let $\f: X\to Y$ be a fibration between two compact ergodic Lie-fibered nilspaces.
Then $\f$ induces a surjective continuous homomorphism $\f_*:\Aut_i^\circ(X)\to\Aut_i^\circ(Y)$
such that
\[
\f_*f.\f(x)=\f(f.x)
\]
for all $x\in X$ and $f\in\Aut_i^\circ(X)$.
\end{thm}

Combining this with Theorems \ref{th:toral} and \ref{th:invlim}, we deduce the following.

\begin{thm}  \label{th:alg-struc}
Let $X$ be a compact ergodic strongly connected nilspace of degree $s$.
Then:
\begin{itemize}
    \item there exists a sequence of connected Lie groups $G^{(n)}$
equipped with filtrations $G^{(n)}_\bullet$ of degree at most $s$ (with $G^{(n)}_i$ also connected for each $i$);
    \item for each $n$ there is a discrete co-compact subgroup $\Gamma^{(n)}$ of $G^{(n)}$ compatible with the filtration; and
    \item for each $\infty>n\ge m$, there are surjective group homomorphisms
      $\psi_{m,n} : G^{(n)} \to G^{(m)}$;
\end{itemize}
such that, letting $X_n:=G^{(n)}/\Gamma^{(n)}$ be the Host--Kra nilspace
associated to $G^{(n)}_\bullet$ and $\Gamma$, the following hold:
\begin{itemize}
  \item for each $n \ge m$, we have $\psi_{m,n}\left(G^{(n)}_i\right) = G^{(m)}_i$ for each $i \ge 0$;
  \item again for each $n \ge m$, we have that $\psi_{m,n}\left(\Gamma^{(n)}\right) \subseteq \Gamma^{(m)}$;
  \item the map
    \begin{align*}
      \f_{m,n} \colon X_n &\to X_m \\
      g\cdot\Gamma^{(n)} &\mapsto \psi_{m,n}(g)\cdot\Gamma^{(m)}
    \end{align*}
    induced by $\psi_{m,n}$, is a fibration; and
  \item $X$ is homeomorphic and isomorphic as a nilspace to the inverse limit $\varprojlim X_n$
given by the inverse system $\f_{m,n}$.
  \end{itemize}
\end{thm}
This result should be thought of as a strengthening of the statement that $X$ is an inverse limit of nilmanifolds: indeed, it states precisely this, but giving much more information concerning the nature of the connecting maps.

As a consequence of these results, it is possible to endow a compact ergodic strongly connected nilspace $X$
with an action of the inverse limit of the groups $\Aut_1^\circ(X_n)$. One might hope to use this information to represent $X$ itself as a homogeneous space of this inverse limit group.
Unfortunately, this fails, because the action need not be transitive in general.
Moreover, there are examples due to Rudolph \cite{R95}, in which $X$ cannot be represented as a homogeneous space
of any nilpotent group.
Nevertheless, it is possible to make some direct statements about the nilspace structure of $X$ on the strength of these results.
We believe there are interesting issues here; however, we will not pursue them presently, though we may return to them in future work.

Let $(H,X)$ be a minimal topological dynamical system as in the setting of Theorem \ref{th:dynamics}.
As we discussed in Section \ref{sc:cubes}, we know that $X/\RP_H^k(X)$ is an ergodic
nilspace of degree $k$.
Using Theorem \ref{th:invlim}, we can write $X/\RP_H^k(X)$ as the inverse limit
of Lie-fibered nilspaces $X_n$. Notice that $H\subseteq\Aut_1(X/\RP_H^k(X))$, by the definition of dynamical cubes.
To conclude that
$X/\RP_H^k(X)$ is a pronilsystem, as claimed in Theorem \ref{th:dynamics}, we only need to show that the action of $H$ descends to $X_n$ for $n$ sufficiently large. This is proved in Section \ref{sc:dyninverse}
using the same circle of ideas as is discussed above.

Moreover, we will obtain the following slightly stronger form of Theorem \ref{th:dynamics}.

\begin{thm}  \label{th:dyn-alg-struc}
Let $(H,X)$ be a minimal topological dynamical system.
Suppose that  $\RP_H^s(X)$ is trivial for some $s$, and that $H$
has a dense subgroup generated by a compact set.

Then:
  \begin{itemize}
\item there exists a sequence of nilpotent Lie groups $G^{(n)}$ of degree at most $s$;
\item for each $n$, there is a continuous homomorphism $\a_n:H\to G^{(n)}$;
\item for each $n$, there is a discrete co-compact subgroup $\Gamma^{(n)}\subseteq G^{(n)}$; and
\item for each $n>m$, there is a continuous homomorphism $\psi_{m,n}:G^{(n)}\to G^{(m)}$;
\end{itemize}
such that
\begin{itemize}
\item $\psi_{m,n}(\Gamma_n)\subseteq\Gamma_m$;
\item $\a_m=\psi_{m,n}\circ \a_n$; and
\item the system $(H,X)$ is isomorphic, as a topological dynamical system, to the inverse limit of
$(H,G^{(n)}/\Gamma^{(n)})$ along the inverse system of maps induced by $\psi_{m,n}$, where $H$
acts on $G^{(n)}/\Gamma^{(n)}$ via $\a_n$.
\end{itemize}
\end{thm}

Finally, we note that Theorem \ref{th:dyn-alg-struc} is also valid in the following slightly more
general setting.
The condition that $\RP_H^s(X)$ is trivial could be replaced by the assumption that $X$
is a nilspace of degree at most $s$ and $H$ acts on $X$ via a continuous homomorphism
$H\to \Aut_1(X)$.
This variant allows the cubespace structure on $X$ to have more cubes than the dynamical cubespace $(X,\{C^k_H(X)\}_{k\in\N})$, 
provided it is still  assumed to be a nilspace of degree $s$.

\subsection{Acknowledgments}

First and foremost we owe gratitude to Bernard Host who introduced us to the subject and to Omar Antol\'\i n Camarena and Bal\'azs Szegedy
whose groundbreaking work \cite{CS12} was a constant inspiration
for us.

We would like to thank Emmanuel Breuillard, J\'er\^ome Buzzi,
Yves de Cornulier, Sylvain Crovisier, Eli Glasner, Ben Green, Bernard Host, Micha\pol{} Rams, Bal\'azs Szegedy, Anatoly Vershik and Benjamin
Weiss for helpful discussions. We are grateful to Pablo Candela, Bryna Kra and Bingbing Liang for a careful reading of preliminary versions.

We are thankful to Pablo
Candela, Diego Gonz\'alez-S\'anchez and Bal\'azs Szegedy for pointing out a gap in an earlier draft. This has been corrected in the current version.
See Section \ref{sc:dyninverse} for a discussion.

We are also grateful to the referee for a careful reading of the paper and for useful suggestions, which greatly improved the presentation of this paper.

\section{Inverse limits in the category of cubespaces}\label{sec:cubeinvlim}

The purpose of this section is to prove Theorem \ref{th:invlim}, the inverse limit theorem.  We recommend the reader review the outline of the proof contained in Section \ref{subsec:invlim} before proceeding with the details.

\subsection{Preliminaries and definitions}

For each metrizable topological space we fix a metric that we always denote by $\dist(\cdot,\cdot)$.
The choices of these metrics may be arbitrary.

The following fact is the starting point of the proof.

\begin{lem}
\label{lem:groupinvlim}
A compact abelian group is the inverse limit of compact abelian Lie groups.
\end{lem}

This allows us to deduce that the structure groups of a nilspace are inverse limits of Lie groups.
In particular, the degree $1$ case of the theorem could be easily deduced from this alone.

\begin{proof}[Proof of Lemma \ref{lem:groupinvlim}]
This is straightforward given Pontryagin duality.  Indeed, it follows easily from the fact that a compact abelian group is a Lie group if and only if its dual group is finitely generated.
See also \cite{Sep07}*{Theorem 5.2(a)}.
\end{proof}

Let $X_\infty^{(\infty)}$ be a compact ergodic nilspace of degree $s$.
(The reason for the double index will become clear in due course, or may be deduced from Section \ref{subsec:invlim}.)
We prove that $X_\infty^{(\infty)}$ is an inverse limit of Lie-fibered nilspaces --
the statement of Theorem \ref{th:invlim} -- by induction on $s$.
If $s=0$, then $X_\infty^{(\infty)}$ is the one point space and the theorem is trivial.

We fix $s\ge 1$, and assume that the theorem holds for nilspaces of degree $(s-1)$.
We denote by $B_\infty=\pi(X_\infty^{(\infty)})$ the $(s-1)$-th canonical factor of $X_\infty^{(\infty)}$.
Here, and throughout the remainder of this section, we abbreviate $\pi_{s-1}$ to $\pi$, as
we will not need to use any of the other canonical projections.
By the induction hypothesis, we may write $B_\infty=\invlim B_m$
for a sequence $\{B_m\}_{m\in\N}$ of Lie-fibered nilspaces, along fibrations $\psi_m \colon B \to B_m$.

We denote the top structure group of $X_\infty^{(\infty)}$ by $A^{(\infty)}$.
Applying Lemma \ref{lem:groupinvlim}, we write $A^{(\infty)}=\invlim A^{(n)}$, where
$\{A^{(n)}\}_{n\in \N}$ is a sequence of Lie groups.
Let $K_n$ denote the kernel of the surjective homomorphism $A^{(\infty)}\to A^{(n)}$.
We define
$X_\infty^{(n)}$ to be the quotient of $X_\infty^{(\infty)}$ under the action of $K_n$
and write $\a^{(n)}:X_\infty^{(\infty)}\to X_\infty^{(n)}$ for the quotient map.

It is not hard to show that $\a^{(n)}$
is a fibration, from which it follows that $X_\infty^{(n)}$ is a nilspace (by \cite{GMV1}*{Remark 7.9}). 

Indeed, let $\l:  \{0,1\}^\ell\sm\{\vec 1\}\to X_\infty^{(\infty)}$ be an $\ell$-corner and $c\in C^\ell(X_\infty^{(n)})$
a cube which is compatible with $\l$ in the sense that
$\a^{(n)}\circ\l=c|_{\{0,1\}^\ell\sm\{\vec 1\}}$.
Let $\wt c\in C^{\ell}(X_\infty^{(\infty)})$ be an arbitrary cube such that $\a^{(n)}(\wt c) = c$ (which exists by definition of the quotient cubespace).
Let $f:\{0,1\}^{\ell}\backslash\{\vec 1\}\to K_{n}$ be the unique configuration such that $f.\wt c|_{\{0,1\}^\ell\sm\{\vec 1\}}=\l$.
It follows from the weak structure theorem that $f$ is a corner in $\mathcal{D}_s(K_n)$,
hence it can be completed to a cube $f'$.
Thus $f'.\wt c$ is a completion of $\l$ that is compatible with $c$.

We will shortly state a technical proposition, which claims the existence of several nilspaces and
maps in between them.
We will see that Theorem \ref{th:invlim} follows from it directly.
In order to formulate the statement more easily, we first introduce some terminology.
Note that these definitions depend on the value of $s$ that has been fixed above.

Let $\f:X\to Y$ be a fibration between two compact ergodic nilspaces of degree $s$.
We claim that there is a unique fibration $\psi \colon \pi(X) \to \pi(Y)$ such that the diagram
\begin{center}
\begin{tikzpicture}
\matrix(m) [matrix of math nodes, row sep=3em, column sep=3em]
{
 X &  Y & & \\
  \pi(X) &  \pi(Y) & & \\
};
\path[-stealth]
(m-1-1) edge node [above] {$\f$} (m-1-2) edge node [left] {$\pi$} (m-2-1)
(m-1-2) edge node [right] {$\pi$} (m-2-2)	
(m-2-1) edge node [above] {$\psi$} (m-2-2)	
;
\end{tikzpicture}
\end{center}
commutes.
Indeed, by the universal property of the canonical factor (see \cite{GMV1}*{Remark 6.7}), 
the composite $\pi \circ \f \colon X \to \pi(Y)$ must factor uniquely through $\pi(X)$, which gives us $\psi$; and by the universal property of fibrations (see  Lemma \ref{lem:universal1}), since $\pi$ and $\psi \circ \pi$ are fibrations, so is $\psi$.
\begin{dfn}
  With this set-up, we call $\psi$ the {\bf shadow} of $\f$.
\end{dfn}

We recall from \cite{GMV1}*{Definition 7.10} 
the notion of $k$-uniqueness for a fibration $\f \colon X \to Y$: this states that if $c_1, c_2 \in C^k(X)$ are such that $\f(c_1) = \f(c_2)$ and $c_1(\o) = c_2(\o)$ for all $\o \ne \vec{1}$, then in fact $c_1 = c_2$.

We introduce the following terminology for the special case that $X$ and $Y$ are nilspaces of degree $s$ and $k=s$.
Recall again that we write $\llcorner^s(x;y)$ to denote the configuration given by $\vec1 \mapsto y$ and $\o \mapsto x$ for all $\o \ne \vec1$.

\begin{dfn}
  \label{dfn:horizontal}
We say that a fibration $\f \colon X \to Y$ between two nilspaces of degree $s$ is {\bf horizontal} if one of the following two equivalent conditions holds:
  \begin{enumerate}
    \item we have $\f(x_1)\neq\f(x_2)$ for any two points $x_1,x_2\in X$ with $\pi(x_1)=\pi(x_2)$ and $x_1\neq x_2$;
    \item for any $x \in X$ with $\f(x) = y$, we have that $\f$ restricts to a bijection $\pi^{-1}(x) \to \pi^{-1}(y)$.
  \end{enumerate}
\end{dfn}

\begin{rmk}\label{rmk:horizontal}
Using the notation and terminology introduced in \cite{GMV1}*{Section 7}, the notion of horizontal fibrations
can also be characterized by one of the following equivalent conditions:
  \begin{enumerate}
    \item[(3)] $\f$ has (relative) $s$-uniqueness;
    \item[(4)] the canonical equivalence relation $\sim_{\f,s-1}$ is trivial (see \cite{GMV1}*{Section 7.2} for a definition); 
    \item[(5)] for all $x_1, x_2 \in X$ such that $\f(x_1) = \f(x_2)$, the configuration $\llcorner^{s}(x_1;x_2)$ is a cube only if $x_1 = x_2$.
\end{enumerate}

We will use the equivalence of (3)--(5) and the definition only later in Section \ref{sc:functor}.
In particular, the proof of Theorem \ref{th:invlim} does not require this.
Hence the reader may safely choose to ignore this remark for the moment.

The equivalence of (3), (4) and (5) is dealt with in \cite{GMV1}*{Section 7.2}. 
Logically, (1) is the same as saying that if $\pi(x_1)=\pi(x_2)$ and $\f(x_1)=\f(x_2)$ then $x_1=x_2$; and by Lemma \ref{lem:alter-canonical}, this is the same as (5).
\end{rmk}

We now state the promised technical proposition.

\begin{prp}\label{prp:invlimlift}
  Let $X_\infty^{(n)}$, $\alpha^{(n)}$, $B_m$, $\psi_m$, etc.~be as above.

  Then there is an increasing sequence $\{M_n\}$ of positive integers such that the following holds.
  For all $n\in\N$ and $m\ge M_n$, there is a compact ergodic nilspace $X^{(n)}_m$ of degree $s$
  such that its $(s-1)$-th canonical factor is $\pi(X^{(n)}_m)= B_m$, and its top structure group is $A_n$.
  Furthermore, there is a horizontal fibration $\f_{m}^{(n)}:X_\infty^{(n)}\to X_m^{(n)}$,
  whose shadow is $\psi_m$.

  In addition, we have the following property.
  If $m_1\le m_2$ and $n_1\le n_2$ are such that $X_{m_1}^{(n_1)}$,  $X_{m_2}^{(n_2)}$ are both
  defined, then the fibres of
  \[
  \f_{m_2}^{(n_2)}\circ\a^{(n_2)}:X_\infty^{(\infty)}\to X_{m_2}^{(n_2)}
  \]
  partition the fibres of
  \[
  \f_{m_1}^{(n_1)}\circ\a^{(n_1)}:X_\infty^{(\infty)}\to X_{m_1}^{(n_1)};
  \]
  more precisely,
  for each $x_1\in X_{m_1}^{(n_1)}$ there is $x_2\in X_{m_2}^{(n_2)}$
  such that $(\f_{m_2}^{(n_2)}\circ\a^{(n_2)})^{-1}(x_2)\subseteq (\f_{m_1}^{(n_1)}\circ\a^{(n_1)})^{-1}(x_1)$.
\end{prp}

Note that the last claim about the fibres of $\f_m^{(n)}\circ\a^{(n)}$ is used to construct fibrations from
 $X_{m_2}^{(n_2)}$ to  $X_{m_1}^{(n_1)}$ with the help of the universal property of fibrations.  Specifically, it asserts the existence of a map of sets $X_{m_2}^{(n_2)} \to X_{m_1}^{(n_1)}$ that makes everything commute; and by these universal properties this will turn out to be continuous and a fibration.

Assuming this proposition, whose proof we will return to below, we now complete the proof of Theorem \ref{th:invlim}.  The deduction is more or less a case of formalizing the previous remark, and we advise against taking the formal details too seriously.

Since $X_{m}^{(n)}$ are Lie fibred nilspaces, it is enough to prove the following.

\begin{prp}\label{pr:inverse-sys}
Let $X_m^{(n)}$, $\alpha^{(n)}$, $B_m$, $\psi_m$, $M_n$, etc.~be as above.
Let $\{\wt M_n\}$ be an increasing sequence such that $\wt M_n\ge M_n$ for all $n$ and
$\wt M_n\to \infty$.

Writing $X_n=X_{\wt M_n}^{(n)}$ and $X_\infty=X_\infty^{(\infty)}$, there is an
inverse system of fibrations $\{\f_{n,l}:X_l\to X_n\}_{\infty\ge l\ge n}$
such that the shadow of $\f_{n,\infty}$ is $\psi_n$ and $X_\infty=\varprojlim X_n$.
\end{prp}

The reader might be wondering what the purpose of the sequence $\{\wt M_n\}$ is, and why we cannot simply take $\wt M_n=M_n$.
The reason is that we will need this extra flexibility in our proof of the dynamical results in Section \ref{sc:dyninverse}.

\begin{proof}
We set $\f_{n,\infty}=\f_{\wt M_n}^{(n)}\circ\a_n$ and construct the fibrations $\f_{n,l}$ for $l<\infty$
by applying Lemma \ref{lem:universal1} (the universal property of fibrations)
with $X=X_\infty=X^{(\infty)}_\infty$,
$Y=X_l=X_{\wt M_l}^{(l)}$ and $Z=X_n=X_{\wt M_n}^{(n)}$.
We verify that $\f_{n,l}\circ\f_{l,o}=\f_{n,o}$.
If $o=\infty$, this already follows from Lemma \ref{lem:universal1}.
If $o<\infty$ then we can write
\[
(\f_{n,l}\circ\f_{l,o})\circ\f_{o,\infty}=\f_{n,l}\circ(\f_{l,o}\circ\f_{o,\infty})=
\f_{n,l}\circ\f_{l,\infty}=\f_{n,\infty}.
\]
Hence $\f_{n,l}\circ\f_{l,o}=\f_{n,o}$
by the uniqueness part of Lemma \ref{lem:universal1} applied with $X=X_\infty$,
$Y=X_o$ and $Z=X_n$.

Finally, now that we have an inverse system, we need to verify that it separates points of $X_\infty$.  Let $x,y\in X_\infty=X_\infty^{(\infty)}$ be distinct points; we need to show that $\f_n(x)\neq\f_n(y)$
if $n$ is sufficiently large.

If $\pi(x)\neq\pi(y)$, then $\psi_{\wt M_n}(\pi(x))\neq\psi_{\wt M_n}(\pi(y))$ if $n$ is sufficiently large.
Since the shadow of $\f_{\wt M_n}^{(n)}$ is $\psi_{\wt M_n}$, we have
\[
\f_{n,\infty}(x)=\f_{\wt M_n}^{(n)}(\a^{(n)}(x))\neq\f_{\wt M_n}^{(n)}(\a^{(n)}(y))=\f_{n,\infty}(y).
\]

If $\pi(x)=\pi(y)$, then there is a unique $a\in A^{(\infty)}$ such that $a.x=y$.
If $n$ is sufficiently large then $a\neq K_n$, hence $\a_n(x)\neq\a_n(y)$.
The claim follows again since $\f_{\wt M_n}^{(n)}$ is horizontal hence injective on fibres of $\pi$.
\end{proof}

\subsection{Straight classes and sections}\label{sec:straight}

We now turn to the proof of Proposition \ref{prp:invlimlift}.

As discussed in Section \ref{subsec:invlim}, we will construct the nilspace $ X_{m}^{(n)}$ as a quotient of $X_{\infty}^{(n)}$ (and indeed, there is essentially no other choice).
To do this, we need to specify a procedure for identifying the fibers $\pi^{-1}(x), \pi^{-1}(y)$ in $X_{\infty}^{(n)}$
for every pair of points $x,y \in B_\infty$ such that $\psi_{m}(x) = \psi_m(y)$.  See Figure \ref{fig:straight} for a pictorial representation.

Moreover, this identification needs to respect the cubespace structure.

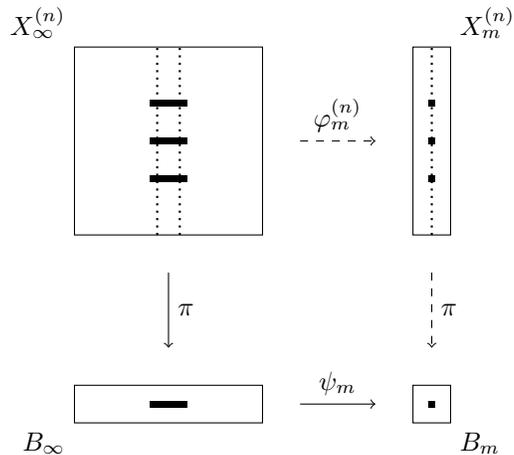
\begin{figure}
\begin{center}
\begin{tikzpicture}[scale=0.5]
  \draw (0,10) node [above left] {$X_\infty^{(n)}$} -- (5,10) -- (5,5) -- (0, 5) -- cycle;
  \draw (0,0) node [below left] {$B_\infty$} -- (5,0) -- (5,1) -- (0,1) -- cycle;
  \draw (10,0) node [below right] {$B_m$} -- (9,0) -- (9,1) -- (10,1) -- cycle;
  \draw (10,10) node [above right] {$X_m^{(n)}$} -- (9,10) -- (9,5) -- (10,5) -- cycle;
  \draw[->] (2.5,4) -- (2.5, 2);
  \node [right] at (2.5,3) {$\pi$};
  \draw[->,dashed] (9.5,4) -- (9.5, 2);
  \node [right] at (9.5,3) {$\pi$};
  \draw[->] (6,0.5) -- (8,0.5);
  \node [above] at (7,0.5) {$\psi_m$};
  \draw[->,dashed] (6,7.5) -- (8,7.5);
  \node [above] at (7,7.5) {$\f_m^{(n)}$};

  \fill (9.42,0.42) rectangle (9.58,0.58);
  \fill (9.42,6.42) rectangle (9.58,6.58);
  \fill (9.42,7.42) rectangle (9.58,7.58);
  \fill (9.42,8.42) rectangle (9.58,8.58);
  \fill (2,0.42) rectangle (3,0.58);
  \fill (2,6.42) rectangle (3,6.58);
  \fill (2,7.42) rectangle (3,7.58);
  \fill (2,8.42) rectangle (3,8.58);

  \draw[thick,dotted] (2.2,5) -- (2.2,10);
  \draw[thick,dotted] (2.8,5) -- (2.8,10);
  \draw[thick,dotted] (9.5,5) -- (9.5,10);

\end{tikzpicture}
\end{center}
  \caption{A schematic of the process of constructing $X_m^{(n)}$.  The map $\f_m^{(n)}$ is supposed to be a horizontal fibration; i.e.~map vertical fibers (left dotted lines) bijectively to vertical fibers (right dotted line).  The three horizontal black lines in $X_\infty^{(n)}$ are three of the desired equivalence classes of the quotient.  In particular, we need to construct an identification between the two left-hand fibers.}
  \label{fig:straight}
\end{figure}

In particular, suppose that $\f_{m}^{(n)}:X_\infty^{(n)}\to X_{m}^{(n)}$ is indeed
a horizontal fibration whose shadow is $\psi_m$.
Then we can make the following observation.
Let $c\in C^{s+1}(B_\infty)$ be a cube all of whose vertices are mapped by $\psi_{m}$ to the same point $b\in B_m$.
Let $x\in X_m^{(n)}$ be a point above $b$, that is $b=\pi(x)$.
Denote by $c_x: \{0,1\}^{s+1}\to X_\infty^{(n)}$ the unique configuration such that $\pi(c_x)=c$ and each
vertex of $\f_{m}^{(n)}(c_x)$ is $x$.
In other words, for the vertices of $c_x$ we pick the unique points above the vertices of $c$ that
are mapped to $x$ by $\f^{(n)}_{m}$.

We claim that $c_x$ is a cube in $X_\infty^{(n)}$.
Indeed, it follows from the weak structure theorem (Theorem \ref{th:weak-structure})
that there is a cube $c'\in C^{s+1}(X_\infty^{(n)})$ with $\pi(c')=c$ and with all
but possibly one of its vertices matching the corresponding vertices of $c_x$; say $c'(\vec 1)$ is the exception.
If we remove the exceptional vertex from $c'$, we obtain a corner all of whose vertices are mapped to $x$
by $\f_{m}^{(n)}$.
Since $\f_{m}^{(n)}$ is a fibration,
this corner can be completed to a cube $c''$ such that $\f_{m}^{(n)}(c'')(\vec 1)=x$.
On the other hand, $\pi(c'')(\vec 1)=c(\vec 1)=\pi(c_x)(\vec 1)$, since $B_m$ is of degree $(s-1)$.
So, we must have that $c''=c_x$.

Motivated by the above observation, we make the following definition.
\begin{dfn}
  Let $X$ be a compact ergodic nilspace, and let $\psi:B_1:=\pi(X)\to B_2$ be a fibration
  onto a nilspace $B_2$.
  We call a set $D\subseteq X$ a {\bf straight $\psi$-class}, if there is a point $b'\in B_2$ such that $\pi(D)=\psi^{-1}(b')$, $D$
  contains exactly one point on the fiber $\pi^{-1}(b)$ for every $b\in\psi^{-1}(b')$,
  and if a configuration $c:\{0,1\}^{s+1}\to D$ is a cube if and only if $\pi(c)$ is a cube.
\end{dfn}
Note that the ``only if'' part of the above condition holds for any set $D$, since $\pi$ is a cubespace morphism.  Also, the reader may wish to verify that this definition is a restatement of the one in Section \ref{subsec:invlim}.

We can summarize the above discussion in the following lemma.

\begin{lem}\label{lem:horizontal-fib}
  If $\f:X\to Y$ is a horizontal fibration between two compact ergodic nilspaces, then the
  inverse images of points under $\f$ are straight $\psi$-classes, where $\psi$ is the shadow of $\f$.
\end{lem}

A key fact, which we will formalize in Propositions \ref{prp:section-exist} and \ref{prp:section-unique} below, is the following.

\begin{fact}
  If $m$ is sufficiently large (in terms of $n$), then each point $x\in X^{(n)}_\infty$ is contained
in a  straight $\psi_m$-class.  Moreover, this class is unique if we also insist that its diameter be small.
\end{fact}

This allows us to define $X^{(n)}_m$ as a quotient of $X_\infty^{(n)}$ by identifying straight classes; i.e., the equivalence classes of the relation $\sim_n^m$ defining $X_m^{(n)}$ are precisely the straight classes given by this fact.

We are not yet done: the condition of Lemma \ref{lem:horizontal-fib} is necessary but not sufficient. We have further work to do to show that this equivalence relation is closed and respects the cubespace structure in the appropriate way to give a horizontal fibration.  However, by Lemma \ref{lem:horizontal-fib} and the above fact, we do know that this quotient is essentially the \emph{only} possible candidate, which is reassuring.

Indeed, the main problem is that the definition of straight classes only carries information about those cubes of $X_\infty^{(n)}$ that
collapse down to a single point in $B_m$ under the map $\psi_m\circ\pi$.  It does not directly tell us anything about general cubes.

Fortunately, our proof of the existence of straight classes in fact yields
further information along these lines.
In order to formalize this, we now make some further definitions.

Let $X$ be a compact ergodic  nilspace and denote by $A$ its top structure group.
For a vertex $\o\in \{0,1\}^{s+1}$ of the discrete cube,
and an element $a\in A$, we write $[a]_\o:\{0,1\}^{s+1}\to A$
for the map which assigns $a$ to $\o$ and the identity element $0$
to all other vertices.
Let $c: \{0,1\}^{s+1}\to X$ be a configuration such that $\pi(c)$ is a cube.
By the weak structure theorem, there is a unique element $a\in A$
such that $[a]_{\vec 0}.c$ is a cube in $X$.
We call this element $a$ the {\bf discrepancy} of $c$ and denote it by $D(c):=a$.

 Let $X$ be a compact ergodic nilspace and let $\psi:B_1:=\pi(X)\to B_2$ be a fibration.
Let $U'\subseteq B_2$ be an open set and put $U=\psi^{-1}(U')$.
We say that a continuous map $\s:U\to X$ is a {\bf straight} $\psi$-{\bf section}
if $\pi\circ\sigma=\Id_U$, and the following holds.
If $c_1,c_2:\{0,1\}^{s+1}\to U$ are two $(s+1)$-cubes of $B_1$ such that $\psi(c_1)=\psi(c_2)$,
then $D(\s(c_1))=D(\s(c_2))$.

We remark that the straightness of a section $\s$ implies that it maps the fibres of $\psi$
onto straight classes.
Indeed, let $c\in C^{s+1}(B_1)$ be a cube contained in a single fibre of $\psi$
and let $c_2$ be a constant cube in the same fibre.
Then $D(\s(c))=D(\s(c_2))=0$, since $\s\circ c_2$ is also a constant cube.
This proves that $\s(c)$ is a cube.

We also note that discrepancy is a continuous function.
Indeed, let $\{c_n\}\subseteq X^{\{0,1\}^{s+1}}$ be a sequence of configurations
converging to a limit $c\in X^{\{0,1\}^{s+1}}$, such that $\pi(c_n)$ is a cube for all $n$ (and hence so is $\pi(c)$).
We show that $\lim D(c_n)=D(c)$.
We assume as we may that $\lim D(c_n)=a$ exists.
Since $C^{s+1}(X)$ is closed, we see that
\[
\lim [D(c_n)]_{\vec 0}.c_n=[a]_{\vec 0}.c
\]
is a cube.
Hence $D(c)=a$ by definition.

We now state two technical propositions, extending the fact about existence and uniqueness of straight classes stated above.

\begin{prp}\label{prp:section-exist}
Let $X$ be an ergodic compact nilspace of degree $s$ and suppose that its top structure group is
a Lie group.  Let $\ve>0$ be given.
Then there is a number $\d>0$ such that the following holds.
Let $\psi: B_1:=\pi(X)\to B_2$ be a fibration onto a nilspace $B_2$
such that $\diam(\psi^{-1}(b_2))<\d$ for all $b_2\in B_2$.

Then for every $c\in C^{s+1}(B_1)$, there is an open set $U_2\subseteq B_2$ such that
$U_1:=\psi^{-1}(U_2)$ contains the vertices of $c$, and such that there exists a straight $\psi$-section $\s: U_1\to X$
such that $\diam(\s(\psi^{-1}(b)))\le\ve$ for all $b\in U_2$.
In particular, each point in $X$ is contained in a straight $\psi$-class
of diameter at most $\ve$.
\end{prp}

\begin{prp}\label{prp:section-unique}
Let $X$ be an ergodic compact nilspace of degree $s$ and suppose that $A$, the top structure group of $X$,
is a Lie group.
Then there is a number $\d>0$ such that the following holds.
Let $\psi: B_1:=\pi(X)\to B_2$ be a fibration onto a nilspace $B_2$.
Let $D_1,D_2\subseteq X$
be two straight $\psi$-classes with $\pi(D_1)=\pi(D_2)$ and
$\diam((D_1\cup D_2)\cap \pi^{-1}(b))\le\d$
for all $b\in\pi(D_1)$.

Then $D_1=a.D_2$ for some $a\in A$.
In particular, $D_1$ and $D_2$ are either equal or disjoint.
\end{prp}

The rest of the section is devoted
to the proofs of Propositions \ref{prp:section-exist} and \ref{prp:section-unique}.
The fact that (local) sections $s: U\to  X$ exist even without any further properties
is already non-trivial, and was established by Gleason in his work on Hilbert's 5th problem.

\begin{thm}[\cite{G50}*{Theorem 3.3}]\label{thm:Gleason}
Suppose that a compact Lie group $A$ acts freely and continuously on a
completely regular topological space $X$.
Denote by $\pi:X\to X/A$ the quotient map under the action of $A$.
Then every point $x\in X/A$ has a neighborhood $U$ such that
there is a local section $\s:U\to X$; that is, a continuous map
satisfying $\pi\circ \s=\Id_U$.
\end{thm}

Although this is a very deep result, we note that the proof in the abelian case
is significantly simpler, which is all that we are using.
See the notes of Tao \cite{T14} for a self-contained treatment.

To prove Proposition \ref{prp:section-exist}, we start with a local section
$\s$ that we obtain from Gleason's theorem.
Then we ``straighten'' it using the action of the structure group $A$;
that is, we choose a suitable continuous function $f:U\to A$
and verify that $x\mapsto f(x).\s(x)$ is a straight $\psi$-section.

To this end, we work out what condition on $f$ implies the straightness of the section $f.s$,
and record it in the following lemma.

We recall the following definition.
Let $f:X\to A$ be a continuous function from a cubespace to an abelian group.
We call the function $\partial^{s+1} f: C^{s+1}(X)\to A$ defined by
\[
\partial^{s+1} f (c):=\sum_{\o\in\{0,1\}^{s+1}}(-1)^{|\o|}f(c(\o))
\]
the $(s+1)$-th {\bf derivative} of $f$.

\begin{lem}\label{lem:straightness}
Let $X$ be an ergodic compact nilspace of degree $s$, and let $\psi:B_1:=\pi(X)\to B_2$ be a fibration
onto a nilspace $B_2$.
Let $U_2\subseteq B_2$ be an open set and put $U_1=\psi^{-1}(U_2)$.
Denote by $A$ the top structure group of $X$.
Let $\s: U_1\to X$ be a section and $f: U_1\to A$ be a continuous map.

Then the section $x\mapsto f(x).\s(x)$ is straight if and only if
\[
\partial^{s+1} f(c_1)-D(\s(c_1))=\partial^{s+1} f(c_2)-D(\s(c_2))
\]
for any two $c_1,c_2\in C^{s+1}(B_1)$ such that $c_1,c_2\subseteq U_1$ and $\psi(c_1)=\psi(c_2)$.
\end{lem}

The above lemma is an immediate corollary of the following.

\begin{lem}\label{lem:discrepancy}
Let $X$ be an ergodic compact nilspace of degree $s$ and denote by $A$ its top structure group.
Consider the configurations $c:\{0,1\}^{s+1}\to X$ and $f:\{0,1\}^{s+1}\to A$.
Then $D(f.c)=D(c)-\partial^{s+1} f$.
\end{lem}

\begin{proof}
By definition $[D(c)]_{\vec 0}.c$ is a cube.
By the weak structure theorem and the identity
\[
\partial^{s+1}(f-[\partial^{s+1}f]_{\vec 0})(c)=0,
\]
the configuration
\[
(f-[\partial^{s+1}f]_{\vec 0}+[D(c)]_{\vec 0}).c=[D(c)-\partial^{s+1} f]_{\vec 0}.(f.c)
\]
is also a cube.
This proves the claim.
\end{proof}

At this point, we have reduced the proof of Proposition \ref{prp:section-exist} to solving a cocycle equation in the sense of Theorem \ref{th:functional}, and therefore the remaining technical core of the proof will be an appeal to that theorem.

However, we caution that we do really need the full technical power of Theorem \ref{th:functional} to make this work, and in particular will apply it to some slightly odd spaces $X$ and $Y$ that are not in general nilspaces.

The remaining work is therefore devoted to setting up these spaces in detail, and to a lot of technical epsilon management.

\begin{proof}[Proof of Proposition \ref{prp:section-exist}]
We will prove only the existence of the straight section, as the claim about straight classes
follows immediately from this.

In fact, we prove the following formally weaker version.
We show that for every cube $c_0\in C^{s+1}(B_1)$,
there is an open set $U_2\subseteq B_2$ such that
$U_1:=\psi^{-1}(U_2)$ contains the vertices of $c_0$,
and there exists a $\d=\d(U_2) > 0$ depending on $U_2$, $X$, $\ve$ such that the
claim of the proposition holds for $U_2$ with this $\d$.
By compactness, this implies the proposition: we may take the minimum of $\d(U_2)$ over a finite collection of open sets whose preimages under $\psi$ cover $C^{s+1}(B_1)$.

Fix a cube $c_0\in C^{s+1}(B_1)$.
By Gleason's theorem, there is an open set $U_{-1}\subseteq B_1$ containing the vertices of $c_0$ and a
continuous section $\s:U_{-1}\to X$.
Indeed, if this were to fail, then let $V\subseteq B_1$ be an open set that admits a continuous section
and that contains the maximal possible number of vertices of $c_0$.
Let $b$ be a vertex of $c_0$ not covered by $V$.
We can assume without loss of generality that $b\notin \overline V$.
Then we can find an open neighborhood $V'$ of $b$ disjoint from $V$ that admits a continuous
section using Gleason's lemma.
Then $V\cup V'$ violates the assumption we made on $V$.

We choose a smaller open set  $U_{0}$ containing $c_0$ such that $\overline U_{0}\subseteq U_{-1}$.
We also fix a number $\tau$ to be set later, depending only on $\ve$ and the number denoted by $\d$
in Theorem \ref{th:functional} applied with the top structure group of $X$ in the role of $A$, and with $\ell = (s+1)$.
We consider $\overline U_{0}$ with the induced cubespace structure, i.e.~
we denote by $C^{s+1}(\overline U_{0})$ the set of $(s+1)$-cubes all of whose vertices lie in $\overline U_{0}$.
We define the function $\rho: C^{s+1}(\overline U_{0})\to A$ by $\rho(c)=D(\s(c))$,
where $D$ is the discrepancy function defined above.

Since $\s$ and $D$ are continuous, $\rho$ is also continuous.
By uniform continuity, there is a number $\d_0>0$ such that $\dist(\rho(c_1),\rho(c_2))<\tau$
whenever $\dist(c_1,c_2)=\max_\omega\dist(c_1(\omega),c_2(\omega))<\d_0$,
and $\dist(\s(x_1),\s(x_2))\le\tau/C$ whenever $\dist(x_1,x_2)\le4\d_0$.
Here $C$ is the implicit constant from Theorem \ref{th:functional}.

We now set $\d$ to be small enough such that $\d\le\d_0$ and
such that the ball of radius $\d$ around each vertex
of $c_0$ is contained inside $U_{0}$.
We construct open sets $U_1\subseteq B_1$ and $U_2\subseteq B_2$ such that $U_1=\psi^{-1}(U_2)$
and $c_0\subseteq U_1\subseteq U_{0}$.
To this end, let $Z$ be the complement of the union of the $\d$-balls around the vertices of $c_0$.
Then $\psi(Z)$ is a compact set, which does not contain any vertex of $\psi(c_0)$, since
$\diam(\psi^{-1}(\psi(c_0(\o))))< \d$ for all $\o\in\{0,1\}^{k+1}$.
We let $U_2$ be the complement of $\psi(Z)$ in $B_2$.
Then $U_1=\psi^{-1}(U_2)$ contains $c_0$ and it is disjoint from $Z$
hence it is contained in the $\d$ neighborhood of the vertices of $c_0$.
This implies $U_1\subseteq U_{0}$.
Moreover, $\diam(\s(\psi^{-1}(b)))\le\tau/C$ for all $b\in \overline{U}_2$ by the assumptions on $\psi$, $\d_0$ and $\d$.

We check that $\rho$ restricted to $C^{s+1}(\overline{U_1})$ is a cocycle.
Let $c_1,c_2,c_3\in C^{s}(U_1)$ be two cubes such that $[c_1,c_2]$, $[c_1,c_3]$ and $[c_2,c_3]$
are all cubes.
Then we know that $[\rho([c_i,c_j])]_{\vec0}.\s([c_i,c_j])$ are cubes in $X$ for $i,j=1,2,3$.
It follows then that
\begin{equation}\label{eq:cube}
([\rho([c_1,c_2])+\rho([c_2,c_3])]_{\vec0}+[\rho([c_2,c_3])]_{(0,\ldots,0,1)}).\s([c_1,c_2])
\end{equation}
is also a cube.
Indeed, this configuration is obtained from $[\rho([c_1,c_2])]_{\vec0}.\s([c_1,c_2])$ by acting
on two adjacent vertices with the same element of the top structure group.

Using that nilspaces have the gluing property (\cite{GMV1}*{Proposition 6.2}) 
for the cube \eqref{eq:cube} and $[\rho([c_2,c_3])]_{\vec0}.\s([c_2,c_3])$, we obtain
that
\[
[\rho([c_1,c_2])+\rho([c_2,c_3])]_{\vec0}.\s([c_1,c_3])
\]
is also a cube.
By the definition of discrepancy this
proves $\rho([c_1,c_3])=\rho([c_1,c_2])+\rho([c_2,c_3])$.

We check that  Theorem \ref{th:functional} applies to the fibration $\psi:\overline U_1\to \overline U_2$.
Both $\overline U_1$ and $\overline U_2$ are considered with the induced cubespace structure
and they are compact, ergodic and have the gluing property.
Moreover  as $B_1$ is an ergodic nilspace of degree $(s-1)$, $\psi$ is a fibration of degree $(s-1)$.
For two cubes $c_1, c_2\in C^{s+1}(\overline{U_1})$ with $\psi(c_1)=\psi(c_2)$, we have $\dist(c_1,c_2)<\d\le\d_0$,
hence $\dist(\rho(c_1),\rho(c_2))<\tau$ that is assumed to be smaller than the number denoted by $\d$ in
Theorem \ref{th:functional}.

We obtain a function $f:U_1\to A$ from Theorem \ref{th:functional} such that
$\rho=\partial^{s+1}f+\wt\rho\circ\psi$ and $\dist(f(x), f(y))\leq C\cdot\tau/C=\tau$
whenever $\psi(x) = \psi(y)$.
This implies that for any two cubes $c_1,c_2\subseteq U_1$ with $\psi(c_1)=\psi(c_2)$ we have
\[
\partial^{s+1}f(c_1)-D(\s(c_1))=-\wt\rho(\psi(c_1))=-\wt\rho(\psi(c_2))=\partial^{s+1}f(c_2)-D(\s(c_2)).
\]
Thus $b \mapsto f(b).\s(b)$ is a straight section on $U_1$  by Lemma \ref{lem:straightness}.

We set $\tau$ to be small enough so that $\tau\le \ve/2$ and such that for any $x\in X$ and $a_1,a_2\in A$
with $\dist(a_1,a_2)<\tau$ we have $\dist(a_1.x,a_2.x)<\ve/4$. As $\diam(\s(\psi^{-1}(b)))<\tau<\ve/2 $,
then $\diam((f.\s)(\psi^{-1}(b)))\le\ve/2+2\ve/4=\ve$ for all $b\in B_2$, as required.
\end{proof}

\begin{proof}[Proof of Proposition \ref{prp:section-unique}]
Write $B_1'=\pi(D_1)=\pi(D_2)$.
Since $D_1$ and $D_2$ are $\psi$-classes, $B_1'$ is the $\psi$ inverse image of
a point $x\in B_2$, hence $B_1'$ is a compact ergodic nilspace and has the gluing property.
We write $B_2'$ for the nilspace whose only point is $x$.

Denote by $\s_i:B_1'\to D_i$ the inverses of $\pi$ restricted to $D_i$ for $i=1,2$.
Let $f:B_1'\to A$ be the function such that $\s_2(x)=f(x).\s_1(x)$ for all $x\in B_1'$.
Fix $c\in C^{s+1}(B_1')$.
Since $D_1$ and $D_2$ are straight classes, $\s_1(c)$ and $\s_2(c)$ are both cubes.
This and the weak structure theorem shows that $\partial^{s+1}f(c)=0$.

On the other hand, we have $\dist(f(x),0)\le\d$ by assumption.
By Theorem \ref{th:functional} applied to $\psi:B_1'\to B_2'$ and the cocycle $0$, $f$ must be constant.
This proves the proposition.
\end{proof}

\subsection{Quotienting by straight classes}\label{sec:quotient by straight class}

Let $X$ be a compact ergodic nilspace of degree $s$ and suppose that $A$, its top
structure group, is a Lie group.
Let $\psi:B_1=\pi(X)\to B_2$ be a fibration onto a nilspace.
If the fibers of $\psi$ have sufficiently small diameter depending on $X$, then
by Propositions \ref{prp:section-exist} and \ref{prp:section-unique},
we know that each point of $X$ is contained in a unique straight
$\psi$-class of small diameter.
We define the equivalence relation $\approx$ on $X$
by requiring that each equivalence class is this unique straight $\psi$-class.

Or purpose in this section is to prove the following proposition.
\begin{prp}\label{prp:quotient by straight class}
Let $X$ be an ergodic compact nilspace of degree $s$ and suppose that $A$, its top
structure group is a Lie group.
There is a number $\d>0$ depending only on $X$ such that the following holds.
Let $\psi:B_1=\pi(X)\to B_2$ be a fibration onto a nilspace.
Suppose that  $\diam(\psi^{-1}(b_2))\le\d$ for all $b_2\in B_2$.

Then the equivalence relation $\approx$  defined above
is  closed.
The quotient cubespace $Y:=X/\approx$ is a nilspace,
whose $(s-1)$-th canonical factor is $B_2$.
The quotient map $\f:X\to Y$
is a horizontal fibration with shadow $\psi$.
\end{prp}

The next lemma will be used to establish that the quotient map $\f$ is a fibration.
(The lemma is equivalent to $Y$ having $(s+1)$-uniqueness.)

\begin{lem}\label{lem:quotient by straight class}
In the setting of Proposition \ref{prp:quotient by straight class}, let $c_1,c_2\in C^{\ell}(X)$
for some $\ell\ge s+1$ and suppose that $c_1(\o)\approx c_2(\o)$ for all vertices
$\o\in\{0,1\}^{\ell}\backslash \{\vec 1\}$.
Then we have $c_1\approx c_2$, that is, we also have $c_1(\vec 1)\approx c_2(\vec 1)$.
\end{lem}

\begin{proof}
It is enough to prove the lemma for $\ell=s+1$, for if $\ell$ is larger we can use the case $\ell=s+1$
for suitable faces of $c_1$ and $c_2$.
Note that $\psi(\pi(c_1))$ and $\psi(\pi(c_2))$  are cubes in $B_2$,
which is a  nilspace of degree $(s-1)$  (as fibrations do not increase the degree),
and we have $\psi(\pi(c_1(\o)))=\psi(\pi(c_2(\o)))$ for $\o\in\{0,1\}^{s+1}\sm\{\vec 1\}$.
Then $\psi(\pi(c_1(\vec1)))=\psi(\pi( c_2(\vec1)))$ also.
Hence there is an element $a\in A$
such that $a.c_1(\vec1)\approx c_2(\vec1)$ and hence
$[a]_{\vec1}.c_1\approx c_2$.
We need to show that $a=0$.

We apply Proposition \ref{prp:section-exist} for the cube $\pi(c_1)$.
Let $U_1\subseteq B_1$ and $U_2\subseteq B_2$ be open sets such that
$\pi(c_1)\subseteq U_1=\psi^{-1}(U_2)$ and let $\s:U_1\to X$ be a
straight section.
Since $\psi(\pi(c_1))=\psi(\pi(c_2))$, this implies that $\pi(c_2)\subseteq U_1$.
Let $f:\{0,1\}^{s+1}\to A$ be such that $f.\s(\pi(c_1))=c_1$.
For each $\o\in\{0,1\}^{s+1}$ the set
\[
([a]_{\vec1}+f)(\o).\s(\psi^{-1}(\psi(\pi(c_2(\o)))))
\]
is a straight class that contains $[a]_{\vec 1}.c_1(\o)$ hence $c_2(\o)$ also.
Thus
$([a]_{\vec1}+f).\s(\pi(c_2))=c_2$.

By Lemma \ref{lem:discrepancy}, we have
\[
0=D(c_1)=D(\s(\pi(c_1)))-\partial^{s+1} f
\]
and
\[
0=D(c_2)=D(\s(\pi(c_2)))-\partial^{s+1} f + (-1)^{s}a.
\]
By the definition of straightness, we have $D(\s(\pi(c_1)))=D(\s(\pi(c_2)))$, hence $a=0$
proving the claim.
\end{proof}

\begin{proof}[Proof of Proposition \ref{prp:quotient by straight class}]
We first show that $\approx$ is closed.
Let $\lim x_i=x$ and $\lim y_i=y$ be  two convergent sequences of points in $X$
such that $x_i\approx y_i$ for all $i$.
By the continuity of $\pi$ and $\psi$, we have
$\psi(\pi(x))=\psi(\pi(y))$.
By Proposition \ref{prp:section-exist}, if $\d$ is sufficiently small, then there is an open set $U\subseteq B_1$
containing both $\pi(x)$ and $\pi(y)$ and a straight section $\s:U\to X$.

We assume without loss of generality that $\pi(x_i),\pi(y_i)\in U$ for all $i$.
Using Proposition \ref{prp:section-unique} we conclude that there are $a_i\in A$ such that $x_i=a_i.\s(\pi(x_i))$
and $y_i=a_i.\s(\pi(y_i))$.
We can assume without loss of generality that $a_i$ converges to an element $a\in A$.
Then by the continuity of $\s$ and the action of $A$, we have $x=a.\s(\pi(x))$ and $y=a.\s(\pi(y))$.
Then $x$ and $y$ are in the same straight class, since $\s(\pi(x))$ and $\s(\pi(y))$ are.
This shows that $x\approx y$ proving closedness.

We denote the projection map $X\to Y$ by $\f$.
We denote by $\pi':Y\to B_2$ the unique map that satisfies $\psi\circ\pi=\pi'\circ \f$,
and we show that it is a cubespace morphism.
Indeed, if $c$ is a cube in $Y$, then there is a cube $c'$ such that $\f(c')=c$,
hence $\pi'(c)=\psi(\pi(c'))$ is a cube, as well.

We show that $\f$ is a fibration.
It is a cubespace morphism by the definition of the cubespace structure on
the quotient cubespace $Y$, so it is left to prove the relative completion
property.
Let $\l$ be an $\ell$-corner in $X$ and let $c$ be a completion
of $\f(\l)$.
We show that there is a completion $c_1$ of $\l$ such that $\f(c_1)=c$.

We first consider the case $\ell\ge s+1$.
In this case, $\l$ has a unique completion $c_1$.
Let $c_2\in C^{\ell}(X)$ be such that $c=\f(c_2)$.
Then $c_1(\o)\approx c_2(\o)$ for all $\o\in\{0,1\}^\ell\backslash \{\vec 1\}$ hence
$c_1\approx c_2$ by Lemma \ref{lem:quotient by straight class}.
Thus $\f(c_1)=\f(c_2)=c$ as required.

Second, we assume that $\ell\le s$.
We use that $\psi$ is a fibration and find a cube $c_0\in C^\ell(B_1)$
that is a completion of $\pi(\l)$ such that $\psi(c_0)=\pi'(c)$.
(Here we used that $\pi'$ is a morphism, a fact that we proved above.)
We set $c_1$ to be the unique configuration such that $\pi(c_1)=c_0$ and $\f(c_1)=c$.
Since the dimension of $c_0$ is at most $s$, any $\pi$ lift of it is a cube, so $c_1$  is a cube in particular.

The fact that $\f$ is a fibration, implies by \cite{GMV1}*{Remark 7.9} 
that $Y$ is a nilspace
and by the universal property that $\pi'$ is a fibration.
We leave it to the reader to verify that $\pi'$ is the $s$-th canonical projection,
$\f$ is horizontal and $\psi$ is its shadow.
\end{proof}

\subsection{Proof of Proposition \ref{prp:invlimlift}}

In this paper, we adopt the convention that we denote by $\pi$ the canonical projection of any nilspace of degree $s$
to its $s-1$'th canonical factor without designating the domain of the map in our notation.
This should not normally cause confusion; however, in this proof we will use the canonical projection of many nilspaces,
and for this reason we temporarily deviate from our usual convention, and write
\[
\pi_m^{(n)}:X_m^{(n)}\to X_m^{(0)}=B_m
\]
for the canonical projection of $X_m^{(n)}$.
This should not be confused with the notation in Section \ref{sc:intro}, where the subscript of $\pi$ designates the degree of the
canonical factor that we project onto.

We take $M_n$ sufficiently large so that $\diam(\psi_m^{-1}(x))<\d$ for all $m\ge M_n$
and $x\in B_m$, where $\d$ is sufficiently small such that
Proposition \ref{prp:quotient by straight class} and  Proposition \ref{prp:section-unique} can
both be applied to $X_\infty^{(n)}$.
Then the existence of the nilspace $X_m^{(n)}$ and the fibration $\f_m^{(n)}$ follows
from  Proposition \ref{prp:quotient by straight class} and it remains to verify
the claim about the inverse images of points under the maps $\f_m^{(n)}\circ \a^{(n)}$.

We take indices $n_2\ge n_1$ and $m_2\ge m_1\ge M_{n_1}$ such that $m_2\ge M_{n_2}$.
Let $x\in X^{(n_2)}_{m_2}$ be a point.
Let $\wt x\in (\f_{m_2}^{(n_2)}\circ\a^{(n_2)})^{-1}(x)\in X_{\infty}^{(\infty)}$ be an arbitrary point and take
$y=\f_{m_1}^{(n_1)}(\a^{(n_1)}(\wt x))$.
We set out to prove that
\[
Z_2:=(\f_{m_2}^{(n_2)}\circ\a^{(n_2)})^{-1}(x)\subseteq Z_1:=(\f_{m_1}^{(n_1)}\circ\a^{(n_1)})^{-1}(y).
\]

We first show that $\a^{(n_1)}(Z_2)\subseteq X^{(n_1)}_\infty$ is a straight $\psi_{m_2}$-class.
Let $b\in D:=\psi_{m_2}^{-1}(\pi_{m_2}^{(n_2)}(x))$ be an arbitrary point.
We show that $\a^{(n_1)}(Z_2)$ contains a unique point in the fibre of $\pi_\infty^{(n_1)}$ above $b$.
Since $(\f_{m_2}^{(n_2)})^{-1}(x)$ is a straight $\psi_{m_2}$-class, it follows that
the points of $Z_2$ in the fibre of $\pi_\infty^{(\infty)}$ above $b$ is a single $K_{n_2}$ orbit.
This projects to a single point under $\a^{(n_1)}$ as $K_{n_2}\subseteq K_{n_1}$.

Let $c:\{0,1\}^{s+1}\to  D$ be a cube.
We show that there is a cube $\wt c:\{0,1\}^{s+1}\to\a^{(n_1)}(Z_2)$ with
$\pi_{\infty}^{(n_1)}(\wt c)=c$.
Since $(\f_{m_2}^{(n_2)})^{-1}(x)$ is a straight $\psi_{m_2}$-class, we can find a
cube $c_1:\{0,1\}^{s+1}\to(\f_{m_2}^{(n_2)})^{-1}(x)$ with
$\pi_\infty^{(n_2)}(c_1)=c$.
Since $\a^{(n_2)}$ is a fibration, there is a cube $c_2:\{0,1\}^{s+1}\to Z_2$
with $\a^{(n_2)}(c_2)=c_1$ and hence $\pi_\infty^{(\infty)}(c_2)=c$.
Thus $\wt c:=\a^{(n_1)}(c_2)$ satisfies the requirements.
This shows that $\a^{(n_1)}(Z_2)$ is indeed a straight $\psi_{m_2}$-class.

We note that $\a^{(n_1)}(Z_1)$ is a straight $\psi_{m_1}$-class.
Then
\[
\a^{(n_1)}(Z_1)\cap(\pi_{\infty}^{(n_1)})^{-1}(D)
\]
 is a straight $\psi_{m_2}$-class and it contains the point $\a^{(n_1)}(\wt x)$, which
is also contained in  $\a^{(n_1)}(Z_2)$.
By Proposition \ref{prp:section-unique} we have hence
\[
\a^{(n_1)}(Z_2)\subseteq \a^{(n_1)}(Z_1)
\]
and then $Z_2\subseteq Z_1$.
This completes the proof of  Proposition \ref{prp:invlimlift}.

\section{Equivariance under translations}\label{sc:functor}

We recall that we denote by $\Aut_i(X)$ the group of $i$-translations of a nilspace $X$.
We endow it with the maximum displacement metric
\[
\dist(f,g)=\max_{x\in X}\{\dist(f(x),g(x))\}.
\]
We denote by $\Aut_i^\ve(X)$ the $\ve$-neighbourhood of the identity in this metric.
If $X$ is a Lie-fibered nilspace then $\Aut_i(X)$ is a Lie group (see \cite{GMV2}*{Theorem 2.18}), hence 
\[
\Aut_i^{\circ}(X)=\langle \Aut_i^\ve(X)\rangle
\]
if $\ve$ is sufficiently small.
The purpose of this section is the proof of Theorem \ref{th:functoriality}, which is
an immediate consequence of the following.

\begin{prp}\label{prp:small-translations}
Let $\f:X\to Y$ be a fibration between two compact ergodic Lie-fibered nilspaces
and let $\ve>0$ be given.
Then there is a $\d>0$ depending only on $X$, $Y$, $\f$ and $\ve$
such that the following holds.
For every $f\in \Aut_i^\d(X)$ there is an $f'\in \Aut_i^\ve(Y)$, and respectively for every
$f'\in \Aut_i^\d(Y)$ there is an $f\in \Aut_i^\ve(X)$, such that
$f'\circ \f=\f\circ f$.
\end{prp}

Note there are really two distinct statements here: a ``pushing forward'' result and a ``pulling back'' one.  The proofs of these will be handled separately and have different flavours.  The ``pulling back'' part is really an existence fact, and will be implied fairly easily by results concerning the existence of translations from \cite{GMV2}.  The ``pushing forward'' part is about proving properties of small translations -- namely, that they are compatible with the fibration in some sense -- and will require a new argument.

In both cases, the proof of this proposition is by induction on $s$, the degree of $X$.
For $s=0$ the claim is trivial, so
we fix $s\ge1$ and
assume that the proposition holds for nilspaces of degree $(s-1)$.

Recall that in Definition \ref{dfn:horizontal} we coined the notion of a \emph{horizontal fibration}.  In Section \ref{sec:verhor} we introduce the complementary notion of \emph{vertical} fibrations, and show that
an arbitrary fibration may be decomposed as a composition
of a vertical and a horizontal one.

The ``pushing forward'' part of Proposition \ref{prp:small-translations} is reasonably straightforward for vertical fibrations, and in fact holds without any smallness assumption.  We handle this in Section \ref{sec:vertical}.  We thereby reduce to the case where $\f$ is horizontal.

We consider this case in Section \ref{sec:horizontal}.  The crucial step is to show that small translations map sets of the form $\f^{-1}(y)$
onto each other.
The key observations are to note that such sets are straight classes, and that, in general, translations map
straight classes onto straight classes.
Hence, the results of the previous section can be exploited to give what we want in the case of small translations.

Finally, we prove the ``pulling back'' result in Section \ref{sec:lifting}.

\subsection{A decomposition of fibrations}\label{sec:verhor}

We recall from Section \ref{sec:straight} that the shadow of a fibration $\f:X\to Y$ between
compact ergodic nilspaces of degree $s$ is the unique fibration $\psi:\pi(X)\to\pi(Y)$ that satisfies $\pi\circ\f=\psi\circ\pi$.
We continue to use our convention that $\pi$ abbreviates $\pi_{s-1}$ as we do not use the
other canonical projections.

We also recall that a fibration $\f$ is called \emph{horizontal} if it has relative $s$-uniqueness; or equivalently, if it is injective on fibres of $\pi$ (see Definition \ref{dfn:horizontal} and the remarks that follow).

Finally, we recall from \cite{GMV1}*{Definition 7.18} 
that a fibration $\f \colon X \to Y$ is called \emph{relatively $k$-ergodic} if whenever $c \colon \{0,1\}^k \to X$ is a configuration such that $\f \circ c \in C^k(Y)$ then $c \in C^k(X)$; i.e.~if all $k$-configurations in $X$ are cubes provided they map to cubes of $Y$.

The complementary notion to a horizontal fibration is as follows.
\begin{dfn}
  We say that a fibration $\f:X\to Y$ between ergodic compact nilspaces of degree $s$ is {\bf vertical} if any of the following equivalent conditions holds:
  \begin{enumerate}
    \item given $x_1, x_2 \in X$ such that $\pi(\f(x_1)) = \pi(\f(x_2))$, we must have $\pi(x_1) = \pi(x_2)$;
    \item the shadow of $\f$ is an isomorphism $\pi(X) \xrightarrow{\sim} \pi(Y)$;
    \item $\f$ is relatively $s$-ergodic;
    \item we have that $x_1 \sim_{\f,s-1} x_2$ for any $x_1, x_2 \in X$ such that $\f(x_1) = \f(x_2)$ (again, see \cite{GMV1}*{Section 7.2} for a definition); 
    \item for any $x_1,x_2 \in X$ such that $\f(x_1) = \f(x_2)$, the configuration $\llcorner^s(x_1;x_2)$ is a cube.
  \end{enumerate}
\end{dfn}
The equivalence of (3), (4) and (5) is covered in \cite{GMV1}*{Section 7.2}. 
Now, (1) states precisely that the shadow of $\f$ is injective, and any fibration is an isomorphism if and only if it is injective; so (1) and (2) say the same.  Clearly, (1) implies (5); and (3) implies (1), since (by Lemma \ref{lem:alter-canonical}) $\pi(\f(x_1)) = \pi(\f(x_2))$ if and only if $\llcorner^s(\f(x_1); \f(x_2))$ is a cube, which holds if and only if $\llcorner^s(x_1; x_2)$ is (assuming (3)), which implies $\pi(x_1) = \pi(x_2)$.

Examples of such fibrations are quotient maps by subgroups of the top structure group.
In fact, it turns out that these are the only examples.

The main result of this section is the following decomposition result.

\begin{prp}\label{prp:verhor}
Let $\f:X\to Y$ be a fibration between two compact ergodic nilspaces of degree $s$.
Then there is a compact ergodic nilspace $Z$, a vertical fibration $\f_v: X\to Z$
and a horizontal fibration $\f_h:Z\to Y$ such that $\f=\f_h\circ\f_v$.
\end{prp}
\begin{proof}
  This is immediate from \cite{GMV1}*{Proposition 7.12}. 
This states that there is a decomposition
  \[
    \f \colon X \xrightarrow{\pi_{\f,s-1}} X / \sim_{\f,s-1} \xrightarrow{g} Y
  \]
  where $\sim_{\f,s-1}$ is the canonical equivalence relation attached to the fibration $\f$, and that both of these maps are fibrations.  It is immediate from the definition of $\sim_{\f,s-1}$ that the quotient map is relatively $s$-ergodic; and it follows from \cite{GMV1}*{Remark 7.9} 
that $Z = X / \sim_{\f,s-1}$ is a nilspace (and also trivially compact and ergodic).
By  \cite{GMV1}*{Proposition 7.12}, 
$g$ is a fibration of degree at most $s-1$, so it has $s$-uniqueness, and hence it is horizontal.
(See also Remark \ref{rmk:horizontal}).
\end{proof}

We make a final remark before proceeding.  We note that the relation $\sim_{\f,s-1}$ is finer than $\sim_{s-1}$ on $X$ (which is immediate from the definitions), and hence $\pi$ factors as
\[
  \pi \colon X \xrightarrow{\pi_{\f,s-1}} \pi_{\f,s-1}(X) \to \pi(X).
\]
Also, the relative structure theorem \cite{GMV1}*{Theorem 7.19} 
states that $X$ admits a free continuous action by a compact abelian group $A_s(\f)$ whose orbits are the fibers of $\pi_{\f,s-1}$.

It is not quite immediate from this, but is nonetheless true and not hard to argue, that when $X$ is a nilspace of degree $s$ this group $A_s(\f)$ may be identified with a closed subgroup of the top structure group $A_s$.  Hence, any vertical fibration is the quotient of $X$ by a subgroup of the top structure group.  However, we will not explicitly need such a result.

\subsection{The case of vertical fibrations}\label{sec:vertical}

In this section, we prove the following.

\begin{prp}\label{prp:translation-vertical}
Let $\f:X\to Y$ be a vertical fibration between two compact ergodic nilspaces.
Then there is a continuous homomorphism $\psi: \Aut_i(X)\to\Aut_i(Y)$
such that $\psi(f)(\f(x))=\f(f(x))$
for all $x\in X$ and $f\in \Aut_i(X)$.
\end{prp}

We begin with a simple lemma which gives a condition for a
translation to descend to a factor through a fibration.

\begin{lem}\label{lem:universal2}
Let $\f:X\to Y$ be a fibration between two compact nilspaces
and let $f\in \Aut_i(X)$ be a translation.
Suppose that for every $y_1\in Y$ there is a $y_2\in Y$ such that
$f(\f^{-1}(y_1))=\f^{-1}(y_2)$.
Then there is a unique translation $f'\in \Aut_i(Y)$ such that $f'\circ\f=\f\circ f$.
\end{lem}
\begin{proof}
We apply Lemma \ref{lem:universal1} for the fibrations $\f: X\to Y$
and $\f\circ f:X\to Y$ and deduce that a unique fibration $f':Y\to Y$ exists
satisfying $f'\circ\f=\f\circ f$.

We show that $f'\in \Aut_i(Y)$.
To this end, we fix a cube $c\in C^{\ell}(Y)$ and let $\wt c$ be a $\f$-preimage of $c$ in $C^{\ell}(X)$.
Let $F\subseteq\{0,1\}^\ell$ be a face of codimension $i$.
Then $[f']_F.c=\f([f]_F.\wt c)$ is a cube, showing that $f'$ is indeed a translation.
\end{proof}

It is possible to verify the condition of Lemma \ref{lem:universal2} directly in the case of vertical fibrations.  Many approaches are possible here; ours is fairly direct, using minimal structure theory.

Recall that we write $\llcorner^k(x;y)$ to denote the $k$-configuration given by $\vec1 \mapsto y$ and $\o \mapsto x$ for all $\o \ne \vec1$.  We also introduce the notation $\square^k(x)$ to denote the constant $k$-cube $\omega \mapsto x$.

\begin{lem}\label{lem:vertical-respects-fibers}
  Suppose $\f \colon X \to Y$ is a vertical fibration between compact ergodic nilspaces of degree $s$, and suppose $f \in \Aut_1(X)$ is an $1$-translation.

  Let $x,x' \in X$ be such that $\f(x) = \f(x')$.  Then $\f(f(x)) = \f(f(x'))$.
\end{lem}
\begin{proof}
  By relative $s$-ergodicity of $\f$, we have that $\llcorner^s(x;x')$ is an $s$-cube, and hence $c = [\llcorner^s(x;x'), \llcorner^s(f(x), f(x'))]$ is an $(s+1)$-cube, since $f$ is a $1$-translation.

  Now let $\tilde{c} = [\square^s(\f(x)), \square^s(\f(f(x)))]$.  This is a cube of $Y$ (by ergodicity); and moreover, $c|_{\{0,1\}^{s+1} \sm \{\vec1\}}$ and $\tilde{c}$ form a compatible $(s+1)$-corner for $\f$.  Since $\f$ is a fibration, we may complete $c|_{\{0,1\}^{s+1} \sm \{\vec1\}}$ to a cube $c'$ such that $\f(c'(\vec1)) = \f(f(x))$.

  But $X$ has $(s+1)$-uniqueness, and hence $f(x') = c(\vec1) = c'(\vec1)$, which gives the result.
\end{proof}

\begin{proof}[Proof of Proposition \ref{prp:translation-vertical}]
  Combining Lemma \ref{lem:universal2} and Lemma \ref{lem:vertical-respects-fibers}, we have shown that for all $f \in \Aut_i(X)$ there exists an unique $\psi(f) \in \Aut_i(Y)$ such that $\psi(f) \circ \f = \f \circ f$.  It is routine to verify that $\psi$ must define a continuous group homomorphism.
\end{proof}

\subsection{Horizontal fibrations}\label{sec:horizontal}

In this section, we prove the following.

\begin{prp}\label{prp:small-translations-horizontal}
Let $\f:X\to Y$ be a horizontal fibration between two compact ergodic Lie-fibered nilspaces
of degree $s$,
and let $\ve>0$ be given.
Then there is a $\d>0$ depending only on $X$, $Y$, $\f$ and $\ve$
such that the following holds.
For every $f\in \Aut_i^\d(X)$ there is $f'\in \Aut_i^\ve(Y)$ such that
$f'\circ \f=\f\circ f$.
\end{prp}

Recall our induction hypothesis,
that Proposition \ref{prp:small-translations} holds for nilspaces of degree at most $(s-1)$.
We fix  a horizontal fibration $\f: X\to Y$ between two compact ergodic Lie-fibered nilspaces of degree $s$
and a parameter $\ve_0>0$.
Denote by $\psi:\pi(X)\to\pi(Y)$ the shadow of $\f$.
We choose a sufficiently small number $\d_0$ such that
 Proposition \ref{prp:small-translations} holds for
$\psi,\pi(X),\pi(Y),\ve_0,\d_0$.

Proposition \ref{prp:small-translations-horizontal} is an immediate consequence of
Lemma \ref{lem:universal2} and the following lemma:

\begin{lem}
If $\d$ is sufficiently small, then
for any point $y\in Y$ there is a point $z\in Y$ such that we have $f(\f^{-1}(y))=\f^{-1}(z)$.
\end{lem}

\begin{proof}
Let $f\in \Aut_i^\d(X)$.
Write $g\in \Aut_i(\pi(X))$ for the shadow of $f$.
If $\d$ is sufficiently small, then $ g\in \Aut_i^{\d_0}(\pi(X))$.
Hence there is a translation $g'\in \Aut_i^{\ve_0}(\pi(Y))$
such that $g'(\psi(x))=\psi(g(x))$
for all $x\in \pi(X)$.
Thus for every $y_1\in \pi(Y)$ there is $z_1\in \pi(Y)$ such that
$g(\psi^{-1}( y_1))=\psi^{-1}(z_1)$.

Recall from Lemma \ref{lem:horizontal-fib} that the inverse images of points under $\f$ are
straight $\psi$-classes in $X$.
We use that $f$ is a cubespace automorphism and the conclusion of the previous paragraph
to deduce that $D_1:=f(\f^{-1}(y))$ is also a straight $\psi$-class.

Now fix a point $x_0\in\f^{-1}(y)$, let $z=\f(f(x_0))$ and let $D_2=\f^{-1}(z)$.
Then $D_1$ and $D_2$ are both straight $\psi$-classes, and $f(x_0)\in D_1\cap D_2$.
Recall that Proposition \ref{prp:section-unique} gives conditions under which two straight classes are always either identical or disjoint; if these conditions hold, we have that $D_1 = D_2$ and hence $f(\f^{-1}(y)) = \f^{-1}(z)$ as required.

So, its suffices to check the hypotheses of Proposition \ref{prp:section-unique}.
We fix a parameter $\kappa>0$ to be specified later.
Since $\f$ is continuous, it follows that for any $x\in X$ we have $\dist(\f(f(x)),\f(x))\le\kappa$
provided $\d$ is sufficiently small.
Hence for all $x\in \f^{-1}(y)$, we have $\dist(\f(f(x)),y)\le \kappa$.
In particular, $\dist(z,y)\le \kappa$, hence $\dist(\f(f(x)),z)\le 2\kappa$ for all $x\in\f^{-1}(y)$.

Let now $x_1\in D_1$, $x_2\in D_2$ with $\pi(x_1)=\pi(x_2)$.
Denote by $\d_1$ the number $\d$ from Proposition \ref{prp:section-unique}.
We aim to show that $\dist(x_1,x_2)<\d_1$ if $\d$ is sufficiently small, which completes the proof.
To that end, we show that if $\kappa$ is sufficiently small, then for any $u,v\in X$ with $\pi(u)=\pi(v)$,
we have that $\dist(\f(u),\f(v))<\kappa$ implies $\dist(u,v)<\d_1$.
Indeed, suppose for contradiction that $(u_n)_{n\in\N}$ and $(v_n)_{n\in\N}$ are two sequences with
$\pi(u_n)=\pi(v_n)$ for all $n$, and $\dist(\f(u_n),\f(v_n))\to 0$ as $n\to\infty$, yet $\dist(u_n,v_n)\ge\d_1$
for all $n$.
We may assume without loss of generality that $u_n$ and $v_n$ are both convergent, and we write $u$
and $v$ for their respective limits.
By continuity of $\pi$ and $\f$, we have $\pi(u)=\pi(v)$ and $\f(u)=\f(v)$, hence $u=v$, because $\f$ is
a horizontal fibration.
However, we also have $\dist(u,v)\ge \d_1$, which is a contradiction.
\end{proof}

\subsection{Pulling back translations}\label{sec:lifting}

All that now remains is the ``pulling back'' component of Proposition \ref{prp:small-translations}, which we now recall.

\begin{prp}\label{prp:lifting}
Let $\f:X\to Y$ be a fibration between two compact ergodic Lie-fibered nilspaces
and let $\ve>0$ be number.
Then there is a $\d>0$ depending only on $X$, $Y$, $\f$ and $\ve$
such that the following holds.
For every $f'\in \Aut^\d_i(Y)$ there is $f\in \Aut^\ve_i(X)$ such that
$f'\circ \f=\f\circ f$.
\end{prp}

We recall the following result from \cite{GMV2}*{Lemmas 3.10 and 3.4} 
that will be used in the proof.

\begin{thm}
Let $X$ be a compact ergodic Lie-fibered nilspace.
Then $\Aut_i^\circ(X)$ acts transitively on each connected component of the equivalence classes
of the $(i-1)$-th canonical equivalence relation $\sim_{i-1}$.

In fact, the following strengthening is true. For every $\ve>0$ there is $\d>0$ such that
the following holds:
for every two points $x_1,x_2\in X$ satisfying $x_1\sim_{i-1} x_2$ and $\dist(x_1,x_2)<\d$,
there is a translation $f\in \Aut_i^\ve(X)$ such that $f(x_1)=x_2$.

Also, the stabilizer $\Stab_x(\Aut_1(X))$ is discrete for any point $x\in X$.
\end{thm}

Essentially this states that, under some topological assumptions, $i$-translations are uniquely characterized by where they send a single point, and we have almost total freedom to choose that point.  So, if $f'$ sends $y$ to $y'$, we can choose $f$ by insisting it maps $x$ to $x'$ for some $x \in \f^{-1}(y)$, $x' \in \f^{-1}(y')$ that we choose; and if we are sufficiently careful, this $f$ will have the required properties.

We now turn to the details.

\begin{proof}[Proof of Proposition \ref{prp:lifting}]
Fix a point $y_1\in Y$.
Let $\ve>0$ be sufficiently small such that $\Stab_{y_1}(\Aut_i^{2\ve}(Y))=\{e\}$.
Let $\d_1>0$ be sufficiently small so that for each $f\in \Aut_i^{\d_1}(X)$ there
is $f'\in \Aut_i^{\ve}(Y)$ such that $f'\circ \f=\f\circ f$.
Let $\d_2$ be sufficiently small so that for every two points
$x_1,x_2\in X$ satisfying $\dist(x_1,x_2)<\d_2$ and $x_1\sim_i x_2$,
there is a translation $f\in \Aut_i^{\d_1}(X)$ such that $f(x_1)=x_2$.
Let $\d_3$ be sufficiently small
so that for the above fixed $y_1\in Y$ and any  point $y_2\in Y$
satisfying $\dist(y_1,y_2)\le \d_3$,
there are $x_1\in\f^{-1}(y_1)$ and  $x_2\in\f^{-1}(y_2)$ so that
$\dist(x_1,x_2)\le\d_2$.

We show that for every $f_0\in \Aut_i^{\d_3}(Y)$, there is $f\in \Aut_i^{\d_1}(X)$
such that $f_0\circ \f=\f\circ f$.
Write $y_2=f_0(y_1)$.
Then $\dist(y_1,y_2)\le \d_3$.
Let $x_1,x_2\in X$ be such that $\f(x_i)=y_i$ for $i=1,2$,  and $\dist(x_1,x_2)\le \d_2$.
Let $f\in \Aut_i^{\d_1}(X)$ be such that $f(x_1)=x_2$.
Let $f'\in \Aut_i^{\ve}(Y)$ be such that $\f\circ f=f'\circ \f$.
Then $f'(y_1)=\f(f(x_1))=y_2=f_0(y_1)$.
Hence $f_0^{-1}f'\in \Stab_{y_1}(\Aut_i^{2\ve}(Y))$ and thus $f_0=f'$.
\end{proof}

\section{Inverse limits in the dynamical category}\label{sc:dyninverse}

The purpose of this section is to prove Theorems \ref{th:dynamics} and \ref{th:dyn-alg-struc}.

Let $(H,X)$ be a minimal topological dynamical system such that $H$ has a dense subgroup generated by a compact
set $K$.
By Theorem \ref{th:dyn-nilspace},
$\RP_H^s(X)$ is a closed $H$-invariant equivalence relation.
For Theorem \ref{th:dynamics}, it is left to show that $(H, X/\RP_H^s(X))$ is the largest
pronilfactor of $(H,X)$ of degree at most $s$.

The main outstanding issue is whether the action of a group of
translations on a nilspace whose top structure group is Lie, descends through a horizontal fibration when the translations are not necessarily small perturbations of the identity.
For vertical fibrations, Proposition \ref{prp:translation-vertical} can be applied, but we will
need a substitute for Proposition \ref{prp:small-translations-horizontal}, which is the following.

\begin{prp}\label{pr:dyninvlim}
Let $X$ be a compact ergodic nilspace of degree $s$ such that its top structure group $A_s(X)$ is Lie.
Let $H<\Aut_1(X)$ be a group of translations,
which has a dense subgroup generated by a compact set.
Then there is a number $\d>0$ depending on $X$ and $H$ such that the following holds.

Let $\f:X\to Y$ be a horizontal fibration to a nilspace $Y$  such that
$\diam(\f^{-1}(y))<\d$ for all $y\in Y$.
Suppose that there is a continuous homomorphism $\bar \psi: \bar H\to \Aut_1(\pi(Y))$
such that $\bar \psi(f)\circ\bar\f=\bar\f\circ f$ for all $f\in \bar H$, where
$\bar H$ is the image of $H$ in $\Aut_1(\pi(X))$ and
$\bar \f:\pi(X)\to\pi(Y)$ is the shadow of $\f$.
Then there is a continuous homomorphism $\psi:H\to \Aut_1(Y)$ such that $\psi(f)\circ\f=\f\circ f$ for all
$f\in H$.
\end{prp}

In an earlier draft, we stated the above result in a stronger form without assuming that the
top structure group of X is Lie, the translation is horizontal and that $H$ factors through the shadow of $\f$.
It was pointed out to us by Pablo Candela, Diego Gonz\'alez-S\'anchez and Bal\'azs Szegedy that this result
does not hold in that generality.
They provide a counterexample in \cite{CGS}*{Example 2.2}.
However, the above stated weaker form is sufficient for our purposes.

We prove Proposition
\ref{pr:dyninvlim} in Section \ref{sc:H-descends}.
Then we prove the part
of Theorem \ref{th:dynamics} that $X/\RP_H^s$ is a pronilsystem
in Section \ref{sc:pronil}.
Then we prove that it is the largest such factor in Section \ref{sc:largest}.
Finally, we prove the remaining parts of Theorem \ref{th:dyn-alg-struc} in Section \ref{sc:dyn-alg-struc}.

\subsection{Proof of Proposition \ref{pr:dyninvlim}}\label{sc:H-descends}

We fix a compact set $K\subseteq H$ that generates a dense subgroup of $H$.
We first show that these translations factor through $\f$.

\begin{lem}\label{lem:compactset}
With the above notation and assumptions in Proposition \ref{pr:dyninvlim},
for every $f_1\in K$ there is a unique $f_2\in \Aut_1(Y)$ such that
$f_2\circ\f=\f\circ f_1$, provided $\d$ is sufficiently small depending on $K$.
\end{lem}

\begin{proof}
In light of Lemma \ref{lem:universal2}, it is enough to show that for any $f\in K$ and $y_1\in Y$
there is $y_2\in Y$ such that $f(\f^{-1}(y_1))=\f^{-1}(y_2)$.

To prove this, we recall that $\f^{-1}(y)$ is a straight
$\bar\f$-class for all $y\in Y$ (see Lemma \ref{lem:horizontal-fib}).
By Proposition  \ref{prp:section-unique}, we know that
there is a number $\d_1>0$ such that each point $x\in X$ is contained in
at most one straight $\bar\f$-class of diameter at most $\d_1$.

We assume as we may that $\d$ is so small that $\diam(f(D))<\d_1$
for any set $D\subseteq X$ with $\diam(D)\le\d$ and $f\in K$.
This is possible, because $K$ is compact and the action is continuous.
We also assume that $\d<\d_1$.
We fix a point $y_1\in Y$, then pick an arbitrary element $x_2\in f(\f^{-1}(y_1))$ and let $y_2=\f(x_2)$.
We observe that $\f^{-1}(y_2)$ and $f(\f^{-1}(y_1))$ are both straight $\bar \f$-classes of diameter at most
$\d_1$.
Moreover, they both contain the point $x_2$, hence they must be equal by Proposition  \ref{prp:section-unique},
as we noted above.
This completes the proof.
\end{proof}

Now we complete the proof of the proposition.
For any translation $f_1\in \Aut_1(X)$, there is at most one translation $f_2\in \Aut_1(Y)$ such that
\be\label{eq:factorsthrough}
f_2\circ\f=\f\circ f_1.
\ee
For $f_1\in K$ the existence of $f_2$ satisfying \eqref{eq:factorsthrough} follows from
Lemma \ref{lem:compactset}.
If $f_1^{(1)},f_1^{(2)}\in \Aut_1(X)$ and $f_2^{(1)},f_2^{(2)}\in\Aut_1(Y)$ satisfy
the analogue of \eqref{eq:factorsthrough}, then $f_1^{(1)}\circ f_1^{(2)}$ and $f_2^{(1)}\circ f_2^{(2)}$
satisfy it, as well.
In addition, if the analogue of \eqref{eq:factorsthrough} holds for
$\{f_1^{(i)}\}_{i\in \N}\subseteq \Aut_1(X)$ and $\{f_2^{(i)}\}_{i\in \N}\subseteq\Aut_1(Y)$ for all $i\in\N$
and $f_1^{(i)}$ uniformly converges to a translation $f_1$, then $f_2^{(i)}$ also converges to
a translation and we have \eqref{eq:factorsthrough}.

It follows from the above observations that a map $\psi: H\to \Aut_1(Y)$ exists such that
$\psi(f)\circ\f=\f\circ f$ for all $f\in H$.
We leave it to the reader to verify that this is also a continuous group  homomorphism.

\subsection{Proof that $X/\RP_H^s$ is a pronilfactor}\label{sc:pronil}

The goal of this section is to prove the claim in Theorem \ref{th:dynamics}
that $X/\RP_H^s$ is a pronilfactor of degree at most $s$.

We recall from Theorem \ref{th:dyn-nilspace} that $X/\RP_H^s(X)$ equipped with its dynamical cubes $C_H^k(X/\RP_H^s)$, is an ergodic nilspace of degree
at most $s$.

By Proposition \ref{pr:inverse-sys}, there is a sequence of Lie-fibred nilspaces $\{X_n\}$  together
with an inverse system of fibrations $\{\f_{m,n}:X_n\to X_m\}$
such that $X/\RP_H^s(X)=X_\infty=\invlim X_n$.

We will show that it is possible to choose
$\{X_n\}$ and $\{\f_{m,n}:X_n\to X_m\}$ in such a way that
the action of $H$ descends to $X_n$ through the fibration $\f_{n,\infty}$.
From this it follows that $(H,X_\infty)$ is indeed pronil
thanks to the following result from \cite{GMV2}*{Corollary 2.20}, 
which is a variant of Theorem \ref{th:Lie-dynamical}.

\begin{thm}\label{th:dyn-Lie2}
Let $(H,X)$ be a minimal topological dynamical system, where $X$ is a compact ergodic Lie-fibred
nilspace of degree $s$ and $H$ acts on $X$ through a continuous group homomorphism $\a:H\to\Aut_1(X)$.

Then $(H,X)$ is a nilsystem.
More specifically, the group $G=\langle \Aut_1^\circ(X), \a(H)\rangle$ is a nilpotent Lie group that acts
transitively on $X$.
Hence $(H,X)$ is isomorphic to $(H,G/\Gamma)$, where $\Gamma$ is the stabilizer in $G$ of an arbitrary point
and $H$ acts through the homomorphism $\a$.
\end{thm}

The above formulation differs slightly from that of  \cite{GMV2}*{Corollary 2.20}, where the larger group $\Aut_1(X)$ is taken in the role of $G$.
However, the (very short) proof of \cite{GMV2}*{Corollary 2.20} only uses the action of $\a(H)$ and $\Aut_1^\circ(X)$ to
show transitivity.

We prove the claim that $H$ descends to $X_n$
by induction on the degree $s$ of the nilspace $X_\infty$.
The $s=1$ case being trivial, we assume that $s>1$ and that the claim holds for $s-1$.

We write $B_\infty=\pi(X_\infty)$ and note that the action of $H$ descends to $B_\infty$
by virtue of Proposition \ref{prp:translation-vertical}.
By abuse of notation, we identify $H$ with its image in $\Aut_1(B_\infty)$
and apply the induction hypothesis to the action of $H$ on $B_\infty$.
Therefore, we have a sequence of nilspaces $\{B_n\}$ and an inverse system
$\{\psi_{m,n}:B_n\to B_m\}$ such that the action of $H$ descends to $B_n$
trough $\psi_{n,\infty}$ for each $n$.

We use the above nilspaces $B_n$ and fibrations $\{\psi_{m,n}:B_n\to B_m\}$
as an input in the construction in Section \ref{sec:cubeinvlim} and use a sufficiently
fast increasing sequence $\wt M_n$ (to be chosen later) in Proposition \ref{pr:inverse-sys}
to construct the sequence of Lie-fibred nilspaces $\{X_n\}$ together
with the inverse system of fibrations $\{\f_{m,n}:X_n\to X_m\}$ we alluded to above.

By Proposition \ref{prp:verhor}, we have a decomposition $\f_{n,\infty}=\f_h\circ\f_v$,
where $\f_v: X\to Z$ is a vertical and $\f_h:Z\to X_n$ is a horizontal fibration and $Z$
is a compact ergodic nilspace.
In fact, this decomposition already appears in the construction of $\f_{n,\infty}$ in Proposition
\ref{pr:inverse-sys}.
Using the notation of that proposition, we just have $Z=X_\infty^{(n)}$, $\f_v=\a_n$ and
$\f_h=\f^{(n)}_{\wt M_n}$.
In particular, we see that $Z$ is Lie fibred and we can make the fibers of $\f_h$ as small as
we wish by making $\wt M_n$ large enough.

We apply Proposition \ref{prp:translation-vertical} to $\f_v$ and Proposition \ref{pr:dyninvlim}
to $\f_h$.
We find that there is a continuous homomorphism $\psi:H\to\Aut_1(X_n)$ such that
$\psi(f)\circ \f_n=\f_n\circ f$ for all $f\in H$.
Therefore, $(H,X_\infty)$ is the inverse limit of the topological dynamical systems $(H,X_n)$
via the inverse system $\{\f_{m,n}\}$.
By Theorem \ref{th:dyn-Lie2}, $(H,X_n)$ is a nilsystem of degree at most $s$ for all $n$,
so this completes the proof
that $X/\RP_H^s$ is a pronilfactor of degree at most $s$

\subsection{}\label{sc:largest}

Let $Y=X/{\sim}$ be a pronilfactor of degree $s$.
We show that $\RP_H^{s}(X)\subseteq\sim$.
We denote by $\f:X\to Y$ the quotient map.
It follows directly from the definition that $\f(c)\in C_H^{s+1}(Y)$
for all $c\in C_H^{s+1}(X)$ and hence $(\f(x),\f(y))\in \RP_H^s(Y)$ for
any pair of points $(x,y)\in\RP_H^s(X)$.
Thus the claim follows from the following lemma.

\begin{lem}
Let $(H,Y)$ be a pronil system of degree $s$.
Then $\RP_H^{s}(Y)$ is trivial.
\end{lem}
\begin{proof}
It is enough to show that $\RP_H^{s}(Y)$ is trivial for a nilsystem $(H,Y)$
of degree $s$.
In this case, the dynamical cubes form a Host--Kra nilspace, which is a nilspace of degree at most $s$, as shown in \cite{GMV1}*{Proposition 2.6}. 
Hence $\RP_H^{s}(Y)=\sim_s$  is indeed trivial.
\end{proof}

\subsection{Proof of Theorem \ref{th:dyn-alg-struc}}\label{sc:dyn-alg-struc}

Let $(H,X)$ be a minimal system such that $\RP_H^s(X)$ is trivial and suppose that $H$ has
a dense subgroup generated by a compact set.

This means that we are in the setting of the proof of Theorem \ref{th:dynamics}.
We have, therefore, a sequence of Lie-fibred nilspaces $\{X_n\}$  together
with an inverse system of fibrations $\{\f_{m,n}:X_n\to X_m\}$ such that $X=X_\infty=\invlim X_n$.
We have already proved that the action of $H$ factors through the fibrations $\f_{n,\infty}$, i.e.~we
have continuous homomorphisms $\a_n:H\to \Aut_1(X_n)$ such that $\a_m=\f_{m,n}\circ\a_n$ for $m<n$.
Moreover, by Theorem \ref{th:dyn-Lie2}, the group $G_n=\langle \Aut_1^\circ(X_n),\a_n(H)\rangle$
act transitively on $X_n$.

We saw in the proof of Theorem \ref{th:alg-struc} that there is a continuous surjective homomorphism
$\psi_{m,n}:\Aut_{1}^\circ(X_n)\to\Aut_1^\circ(X_m)$ such that
\[
\psi_{m,n}(g).\f_{m,n}(x)=\f_{m,n}(g.x)
\]
for all $g\in\Aut_1^\circ(X_n)$ and $x\in X_n$.
We can extend $\psi_{m,n}$ to a homomorphism $G_n\to G_m$ by taking
$\psi_{m,n}(\a_n(h))=\a_m(h)$.
We leave it to the reader to verify that this extension satisfies the properties claimed in Theorem \ref{th:dyn-alg-struc}.

Finally, we note that the proof is valid in the slightly more general setting, when $X$ is an arbitrary
compact ergodic nilspace and $H$ acts on $X$ via a continuous homomorphism $H\to \Aut_1(X)$.
(The action has to be continuous and minimal, and $H$ has to contain a compactly generated dense subgroup.)
That is to say, the proof does not require that the cubespace structure on $X$ is defined using the dynamical
construction; it may be larger.

\appendix\section{}

The purpose of this appendix is to prove the following result stated already in \cite{GMV2}*{Theorem 1.1}.

\begin{thm}\label{th:top-conditions}
Let $X$ be a compact, ergodic nilspace of degree $s$.
Suppose $X$ is locally connected and has finite Lebesgue covering
dimension, and that $C^n(X)$ is connected for all $n$.

Then $X$ is isomorphic to a nilmanifold $G/\Gamma$.
That is, there exists a filtered connected Lie group $G_\bullet$,
a discrete co-compact subgroup $\Gamma\subseteq G$,
and a homeomorphism $\f:X\leftrightarrow G/\Gamma$ that
identifies the cubes $C^k(X)$ with the Host--Kra cubes
$\HK^k(G_\bullet)/\Gamma$.
\end{thm}

The structure theorems of this paper apply to the nilspace $X$.
In particular, by virtue of Theorem \ref{th:alg-struc} we can represent $X$ as an inverse limit of
Host--Kra nilmanifolds.
We only need to show that the sequence of approximating nilmanifolds stabilizes under the
topological hypotheses of Theorem \ref{th:top-conditions}.

We recall the notation from Theorem \ref{th:alg-struc}.
We have that $X$ is the inverse limit of Host--Kra nilmanifolds $X_n$ via an inverse system of
fibrations $\f_{m,n}:X_n\to X_m$.

For each $n$, the group $G^{(n)}=\Aut_1^{\circ}(X_n)$ is a connected nilpotent Lie group equipped with
the filtration $G^{(n)}_\bullet=\Aut_\bullet^{\circ}(X_n)$ of degree at most $s$.
Each $G^{(n)}$ contains a discrete co-compact subgroup $\Gamma^{(n)}$ that is compatible with the filtartion.
The cubespace structure on $X_n=G^{(n)}/\Gamma^{(n)}$ is the Host--Kra structure
arising from the filtration $G^{(n)}_\bullet$ and the discrete co-compact subgroup $\Gamma^{(n)}$ as defined in
Section \ref{sc:HK}.

For each $n\ge m$ there is a  (smooth) surjective homomorphism $\f'_{m,n}:G^{(n)}\to G^{(m)}$
such that $\f'_{m,n}(\Gamma^{(n)})\subseteq \Gamma^{(m)}$ and
$\f_{m,n}(g\cdot\Gamma^{(n)})=\f'_{m,n}(g)\cdot\Gamma^{(m)}$.

We aim to prove that the map $\f_{m,n}:X_n\to X_m$ is trivial (i.e. bijective) if $m$ and $n$ are large enough.
This proves that $X$ is isomorphic to $X_n$ for $n$ large, thereby proving Theorem \ref{th:top-conditions}.

We begin by recalling some facts about the universal covers $\wt G^{(n)}$ of the Lie groups $G^{(n)}$
in Section \ref{sc:simply-connected}.
Then we show  in Section \ref{sc:dimension} that the assumption that $X$
has finite Lebesgue covering dimension implies that the dimension
of $\wt G^{(n)}$ must be bounded.
This  implies that the sequence of Lie groups $\wt G^{(n)}$ stabilizes.
Finally, we show in Section \ref{sc:loc-connected} that the sequence $\wt \Gamma^{(n)}$ (the preimage
of $\Gamma^{(n)}$ in $\wt G^{(n)}$) also stabilizes.

\subsection{The universal cover of \texorpdfstring{$G^{(n)}$}{G(n)}}\label{sc:simply-connected}

We refer to \cite{Warner} as a general reference on Lie groups and differential geometry.
We recall from \cite{Warner}*{Theorem 3.25} that the universal cover $\wt G^{(n)}$ of $G^{(n)}$
can be endowed with a unique Lie group structure such that the covering map $\tau:\wt G^{(n)}\to G^{(n)}$
is a homomorphism.

We write $\wt \Gamma^{(n)}=\tau^{-1}(\Gamma^{(n)})$, which is a discrete co-compact subgroup and
note that the map $g\cdot\wt\Gamma^{(n)}\mapsto \tau(g)\cdot\Gamma^{(n)}$ identifies the quotients
$\wt G^{(n)}/\wt \Gamma^{(n)}$ and $G^{(n)}/\Gamma^{(n)}$.

One could prove a version of Theorem \ref{th:alg-struc} with $\wt G^{(n)}$ in place of $G^{(n)}$, but we only need the
following standard fact.

\begin{lem}\label{lm:simply-connected}
For each $n\ge m$ there is a surjective homomorphism $\wt\f_{m,n}:\wt G^{(n)}\to\wt G^{(m)}$ such that
$\tau\circ\wt\f_{m,n}=\f_{m,n}'\circ\tau$.
\end{lem}

We note that $\wt\f_{m,n}(\wt\Gamma^{(n)})\subseteq\wt\Gamma^{(m)}$, hence $\wt\f_{m,n}$
induces a map $\wt G^{(n)}/\wt\Gamma^{(n)}\to\wt G^{(m)}/\wt\Gamma^{(m)}$.
It is easy to see that this map coincides with $\f_{m,n}$.

We include a proof for the reader's convenience.

\begin{proof}
We denote by $\fg^{(n)}$ the Lie algebra of $G^{(n)}$.
This can be identified by the Lie algebra of $\wt G^{(n)}$ so that
$d\tau:\fg^{(n)}\to\fg^{(n)}=\Id$.

By \cite{Warner}*{Theorem 3.14}, the differential $d\f_{m,n}':\fg^{(n)}\to\fg^{(m)}$ of the homomorphism
$\f_{m,n}'$ is a Lie algebra homomorphism.
By \cite{Warner}*{Theorem 3.27}, there is a Lie group homomorphism $\wt \f_{m,n}:\wt G^{(n)}\to \wt G^{(m)}$
such that $d\wt\f_{m,n}=d\f_{m,n}'$.

We observe that
\[
d(\f_{m,n}'\circ \tau)=d\f_{m,n}'\circ \Id =\Id\circ d\wt\f_{m,n}= d(\tau\circ\wt\f_{m,n}).
\]
By \cite{Warner}*{Theorem 3.16} this implies that $\tau\circ\wt\f_{m,n}=\f_{m,n}'\circ\tau$.

Since $\f_{m,n}'$ is surjective, $d\f_{m,n}$ is also surjective, which implies that
$\wt\f_{m,n}$ is open.
In addition,  $\wt G^{(m)}$ is connected, hence $\wt\f_{m,n}$ is surjective.
\end{proof}

\subsection{\texorpdfstring{$G^{(n)}$}{G(n)} stabilizes}\label{sc:dimension}

The purpose of this section is to show the following.

\begin{prp}\label{pr:bdd-dim}
For each $n$, $\dim\wt G^{(n)}$ is at most the Lebesgue
covering dimension of $X$.
\end{prp}

We need two facts from dimension theory.
First, the Lebesgue covering dimension of an open subset of a manifold equals its ordinary dimension
(see \cite{E78}*{Theorem 1.8.2}).
Second, the Lebesgue covering dimension is monotone with respect to inclusion for separable metric spaces.
This property is proved in \cite{E78}*{Theorem 1.1.2} for the small inductive dimension, but in the
setting of separable metric spaces the two notions of dimension coincide (see \cite{E78}*{Theorem 1.7.7}).

Our objective is to show that for each $n$, a sufficiently small open neighbourhood of identity in
$\wt G^{(n)}$ can be embedded in $X$.
This proves our claim in light of the above remarks.

We begin with the following lemma.

\begin{lem}\label{lm:section-nilpotent}
For each $n$, there is an injective continuous map $\psi_n:\wt G^{(n)}\to\wt G^{(n+1)}$ such that
$\f_{n,n+1}\circ \psi_n=\Id$.
\end{lem}

We denote by $\exp:\fg^{(n)}\to \wt G^{(n)}$ the exponential map and note that it is a homeomorphism, since
$\wt G^{(n)}$ is a connected simply connected nilpotent Lie group (see \cite{Knapp-beyond}*{Theorem 1.127}).

\begin{proof}
Since $d\wt\f_{n,n+1}:\fg^{(n+1)}\to \fg^{(n)}$ is a surjective linear map, there is an injective continuous map
$h: \fg^{(n)}\to \fg^{(n+1)}$ such that $d\wt\f_{n,n+1}\circ h=\Id$.
We define $\psi_n=\exp\circ h \circ \exp^{-1}$.

By \cite{Warner}*{Theorem 3.32}, we have
\[
\f_{n,n+1}=\exp\circ d\wt\f_{n,n+1}\circ \exp^{-1}.
\]
Thus
\[
\f_{n,n+1}\circ \psi_n=\exp\circ d\wt\f_{n,n+1}\circ \exp^{-1}\circ\exp \circ h \circ\exp^{-1}=\Id.
\]
\end{proof}

\begin{lem}\label{lm:dfn-f}
For each $n$ and $g\in \wt G^{(n)}$ there is a unique point $x\in X$ such that
\begin{equation}\label{eq:dfn-f}
\f_{m,\infty}(x)=\psi_{m-1}\circ\ldots\circ\psi_n(g)\cdot\wt\Gamma^{(m)}
\end{equation}
for all $m>n$.
\end{lem}

\begin{proof}
By the definition of  inverse limits we only need to verify that for all $m>k\ge n$ we have
\[
\f_{k,m}\Big(\psi_{m-1}\circ\ldots\circ\psi_n(g)\cdot\wt\Gamma^{(m)}\Big)
=\psi_{k-1}\circ\ldots\circ\psi_n(g)\cdot\wt\Gamma^{(k)},
\]
which is immediate from the definition of the maps $\psi_i$.
\end{proof}

\begin{proof}[Proof of Proposition \ref{pr:bdd-dim}]
We fix $n$ and choose a sufficiently small open neighborhood $U$ of the identity in $\wt G^{(n)}$
such that $\overline U\to \wt G^{(n)}/\wt\Gamma^{(n)}$ is injective.
We define the map $f: \overline U\to X$ such that $f(g)$ is the unique point $x\in X$ that satisfies
\eqref{eq:dfn-f}.
It follows easily from the definition of the inverse limit topology that $f$ is continuous.
Moreover, $f$ is injective, because $\f_{n,\infty}\circ f$ is injective.
Since $\overline U$ is compact, $f$ embeds $U$ in $X$, which proves our claim.
\end{proof}

\subsection{\texorpdfstring{$\Gamma^{(n)}$}{Gamma(n)} stabilizes}\label{sc:loc-connected}

In the previous section we proved that $\dim \wt G^{(n)}$  is bounded by the Lebesgue covering dimension
of $X$.
Removing an initial segment of the sequence, we can assume that $\dim \wt G^{(n)}$ is constant.
Then $d\wt \f_{m,n}:\fg^{(n)}\to\fg^{(m)}$ is invertible, hence $\wt \f_{m,n}$ is an isomorphism.
We can thus replace all $\wt G^{(n)}$ by isomorphic copies and assume that $\wt G=\wt G^{(n)}=\wt G^{(m)}$
and $\wt\f_{m,n}=\Id$ for all $m,n$.

Our objective in this section is to show that the sequence $\wt\Gamma^{(n)}$ also stabilizes.

\begin{prp}\label{pr:Gamma-stab}
If $m$ and $n$ are sufficiently large then $\wt\Gamma^{(n)}=\wt\Gamma^{(m)}$.
\end{prp}

\begin{proof}
Suppose to the contrary that the claim of the proposition fails.

Let $U$ be a sufficiently small open neighborhood of the identity in $\wt G$ such that
$ U^{-1}U\cap \Gamma^{(1)}=\{\Id\}$,
i.e. $U\to \wt G/\wt\Gamma^{(1)}$ is injective.

Let $D$ be an open connected subset of $\f_{1,\infty}^{-1}(U\cdot \wt\Gamma^{(1)})$, which exists by the
local connectedness of $X$.
Since $D$ is open, there is a number $n$ and an open subset $V\subseteq X_n$ such that
$\f_{n,\infty}^{-1}(V)\subseteq D$.

Let $m> n$ be such that $\wt\Gamma^{(m)}\subsetneqq\wt\Gamma^{(n)}$.
Let $\g_1,\ldots, \g_k$ be a system of representatives for left cosets of $\wt\Gamma^{(m)}$
in $\wt\Gamma^{(1)}$.
Then
\[
\f_{1,m}^{-1}(U\cdot \wt\Gamma^{(1)})=U\cdot\g_1\cdot\wt\Gamma^{(m)}\stackrel{\circ}{\cup}
\ldots\stackrel{\circ}{\cup}
U\cdot\g_k\cdot\wt\Gamma^{(m)}.
\]

Each of the sets $U\cdot\g_i\cdot\wt\Gamma^{(m)}$ is open and hence closed in
$\f_{1,m}^{-1}(U\cdot \wt\Gamma^{(1)})$.
Since $\f_{m,\infty}(D)$ is connected, it must be contained in one of the sets $U\cdot\g_i\cdot\wt\Gamma^{(m)}$.

Let $g\cdot \wt\Gamma^{(m)}\in \f_{n,m}^{-1}(V)$ be an arbitrary point
and let $\g\in \wt\Gamma^{(n)}\backslash\wt\Gamma^{(m)}$.
We note that
\[
\f_{n,m}(g\cdot \wt\Gamma^{(m)})=\f_{n,m}(g\cdot\g \cdot\wt\Gamma^{(m)})\in V,
\]
hence
\[
g\cdot \wt\Gamma^{(m)},\;g\cdot\g \cdot\wt\Gamma^{(m)}\in \f_{m,\infty}(D).
\]

Thus there is some $i$ such that
\[
g\cdot \wt\Gamma^{(m)},\;g\cdot\g \cdot\wt\Gamma^{(m)}\in U\cdot\g_i\cdot\wt\Gamma^{(m)}.
\]
Then there are elements $u_1,u_2\in U$ and $\b_1,\b_2\in\wt\Gamma^{(m)}$ such that $g=u_1\g_i\b_1$
and $g\g=u_2\g_i\b_2$.
Then
\[
u_1\g_i\b_1\g=u_2\g_i\b_2
\]
and
\[
u_2^{-1}u_1=\g_i\b_2\g^{-1}\b_1^{-1}\g_i^{-1}\in U^{-1}U\cap \wt\Gamma^{(1)}.
\]
Hence $\g_i\b_2\g^{-1}\b_1^{-1}\g_i^{-1}=\Id$.
This in turns gives $\b_1^{-1}\b_2=\g$, which contradicts $\g\notin \wt\Gamma^{(m)}$
and completes the proof.
\end{proof}

\subsection{A dynamical analogue}\label{sc:proof-dyn-alg-struc}
The purpose of this section is to prove Theorem \ref{th:Lie-dynamical}, which is analogous to
Theorem \ref{th:top-conditions} in the dynamical setting.

By Theorem \ref{th:dyn-alg-struc}, there is a sequence of Lie-fibred nilspaces
$\{X_n\}$ together with an inverse system of fibrations $\{\f_{m,n}:X_n\to X_m\}$ and continuous
homomorphisms $\a_n: H\to\Aut_1(X_n)$ such that
$(H,X)=\invlim (H,X_n)$, where $H$ acts on $X_n$ through $\a_n$.

We aim to show that $\f_{m,n}$ is a homeomorphism if $m,n$ are sufficiently large and then we can deduce
that $(H,X)$ is a nilsystem, e.g. from Theorem \ref{th:dyn-Lie2}.

We begin with an observation about connected components.

\begin{lem}\label{lm:components}
$X$ has finitely many connected components, which are both open and closed.
\end{lem}
\begin{proof}
Since $X$ is locally connected, every point of it has a connected open neighborhood.
Since $X$ is compact, we can cover it with finitely many connected open sets $U_1,\ldots,U_n$.

If $A\subseteq X$ is open and closed then either $U_i\subseteq A$ or $A\cap U_i=\varnothing$ for each $i$.
Hence every closed and open subset of $X$ is a union of some $U_i$.
There are at most $2^n$ such sets.
This implies that the connected component of a point $x$ in $X$ is the intersection of all closed and open subsets
that contain $x$, and this is a finite intersection.
The claim follows.
\end{proof}

It follows from the lemma that the number of connected components of $X_n$ is bounded.
We may assume therefore that the number of connected components is the same along the sequence.

We fix an arbitrary point $x_0\in X$ and denote by $X_n^\circ$ the connected component of $\f_{n,\infty}(x_0)$
in $X_n$.
\begin{lem}\label{lm:restr-homeo}
The map $\f_{m,n}|_{X_n^\circ}$ is a homeomorphism to $X_m^\circ$ if $m$ and $n$ are
large enough.
\end{lem}

We claim that this is enough for the proof of Theorem \ref{th:Lie-dynamical}.
We only need to show that $\f_{m,n}$ is injective for $m,n$ large enough.
We know from the lemma that the restriction of $\f_{m,n}$ to the connected components are injective, so we just
need to rule out that the image of two different connected components have a non-empty
intersection.
Since $\f_{m,n}$ maps connected components into connected components, $\f_{m,n}(X_n)$
has fewer connected components than $X_n$ if $\f_{m,n}$ is not injective.
This is not possible, because $X_n$ and $X_m$ have the same number of connected components
and $\f_{m,n}$ is onto.

\begin{proof}[Proof of Lemma \ref{lm:restr-homeo}]
We recall from Theorem \ref{th:dyn-alg-struc} that there are nilpotent Lie groups $G_n$
acting on $X_n$ and homomorphisms $\psi_{m,n}:G_n\to G_m$ such that
$\f_{m,n}(g.x)=\psi_{m,n}(g).\f_{m,n}(x)$ for all $g\in G_n$ and $x\in X_n$.

We denote by $G_n^\circ$ the connected component of the identity in $G_n$.
We note that $G_n^\circ$ acts transitively on $X_n^\circ$.
One way to see this is the following argument.
The Lie group $G_n$ has at most countably many connected components, hence the orbit of $\f_{n,\infty}(x)$
under one of the cosets of $G_n^\circ$ must have non-empty interior by the Baire category theorem.
This implies that the orbit of $\f_{n,\infty}(x)$ under $G_n^\circ$ is open.
Since it is also connected, it must equal the connected component $X_n^\circ$.

Now the proof can be completed by the same argument as in the proof of Theorem \ref{th:top-conditions}
explained in the previous sections.
\end{proof}


\bibliography{invlim2}
\bibliographystyle{abbrv}

\end{document}